\documentclass[11pt]{article}

\usepackage{amsmath,amssymb,amsthm,mathtools}
\usepackage{graphicx}
\usepackage{enumitem}
\usepackage{geometry}
\usepackage{comment}
\usepackage[colorlinks=true,
  linkcolor=blue,
  citecolor=blue,
  urlcolor=blue]{hyperref}

\DeclareMathOperator{\diam}{diam}

\newtheorem{theorem}{Theorem}
\newtheorem{definition}[theorem]{Definition}
\newtheorem{proposition}[theorem]{Proposition}
\newtheorem{lemma}[theorem]{Lemma}
\newtheorem{corollary}[theorem]{Corollary}

\theoremstyle{remark}
\newtheorem{remark}{Remark}[section]
\newtheorem{example}{Example}[section]

\title{Projection Theorems for $\Phi$--Intermediate Dimensions}
\author{Lara Daw and Najmeddine Attia}
\date{}

\begin{document}
\maketitle
\begin{abstract} 
$\Phi$--intermediate dimensions interpolate between Hausdorff and box--counting dimensions by restricting admissible coverings to scale windows of the form $[\Phi(r),\,r]$. Using a family of $\Phi$--dependent kernels, we develop a potential-theoretic framework that characterizes these dimensions in terms of capacities and leads to associated $\Phi$--dimension profiles. This framework provides effective tools for obtaining lower bounds from uniform potential estimates. As an application, we prove Marstrand--Mattila type projection theorems, showing that for $\gamma_{n,m}$--almost all $m$--dimensional subspaces $V$, the $\Phi$--intermediate dimensions of $\pi_V E$ coincide with deterministic profile values depending only on $E$ and $m$. We also discuss consequences for continuity at the Hausdorff end-point and for the box dimensions of typical projections.


\end{abstract}

\section{Introduction}

The behaviour of fractal dimensions under orthogonal projections is a
central topic in geometric measure theory.
In his seminal 1954 work, Marstrand~\cite{Marstrand1954} proved that for
any Borel set $E\subset\mathbb{R}^2$, the Hausdorff dimension of its
projection onto almost every line satisfies
\begin{equation}\label{eq:Marstrand}
\underline{\dim}_H(\pi_V E)=\min\{\underline{\dim}_H E,\,1\},
\end{equation}
where $\pi_V$ denotes orthogonal projection onto the line $V$.
Kaufman~\cite{Kaufman1968} later provided a potential--theoretic proof of
this result.
In higher dimensions, Mattila~\cite{Mattila1975} extended Marstrand’s
theorem, showing that if $E\subset\mathbb{R}^n$ and $1\le m<n$, then
\[
\underline{\dim}_H(\pi_V E)
=
\min\{\underline{\dim}_H E,\,m\}
\qquad
\text{for $\gamma_{n,m}$--almost all $V\in G(n,m)$},
\]
where $\gamma_{n,m}$ denotes the rotation--invariant probability measure
on the Grassmannian $G(n,m)$.
These results establish the almost--sure invariance of Hausdorff
dimension under orthogonal projections, up to the dimension of the
target space.

For coarser notions of dimension, projection behaviour is more delicate.
Falconer and Howroyd~\cite{FalconerHowroyd1997} showed that the lower and
upper Minkowski (box--counting) dimensions of $\pi_V E$ are almost surely
constant and can be expressed in terms of deterministic
\emph{dimension profiles} depending only on $E$ and $m$.
These profiles were originally defined implicitly, limiting their
practical applicability.
A more explicit and analytically tractable formulation was later
introduced by Falconer~\cite{Falconer2019} via capacities, providing a
functional--analytic framework for projection results at the level of
box and packing dimensions.

To interpolate between Hausdorff and Minkowski dimensions, Falconer,
Fraser and Kempton~\cite{FFK2019} introduced the intermediate dimensions
$\underline{\dim}_\theta E$ for $0\le\theta\le1$, defined by restricting
admissible covers to sets with diameters in the range
$[r^{1/\theta},\,r]$.
These dimensions satisfy
$\underline{\dim}_0 E=\underline{\dim}_H E$ and
$\underline{\dim}_1 E=\underline{\dim}_B E$, and form a monotone family
interpolating between the two.
They are bi--Lipschitz invariant and continuous on $(0,1]$, though
discontinuity at the Hausdorff end--point $\theta=0$ may occur.
Intermediate dimensions are particularly effective for sets exhibiting a gap between Hausdorff and box dimensions. They have since become a central tool for analysing multi‑scale and non‑uniform fractal structure.

A significant advance was made by Burrell, Falconer and
Fraser~\cite{BFF2021}, who developed a potential--theoretic
characterisation of intermediate dimensions and proved an almost--sure
projection theorem in this setting.
Specifically, they showed that for any Borel set
$E\subset\mathbb{R}^n$ and each $m\in\{1,\dots,n-1\}$ there exist
deterministic profiles
$\underline{\dim}_\theta^m E$ and $\overline{\dim}_\theta^m E$ such that
\[
\underline{\dim}_\theta(\pi_V E)
=
\underline{\dim}_\theta^m E,
\qquad
\overline{\dim}_\theta(\pi_V E)
=
\overline{\dim}_\theta^m E,
\]
for all $\theta\in[0,1]$ and for $\gamma_{n,m}$--almost all
$V\in G(n,m)$.
This framework simultaneously recovers the Marstrand--Mattila theorem
and the Falconer--Howroyd box--dimension result as limiting cases.

\medskip

The power--law restriction $r^{1/\theta}$ inherent in the definition of intermediate dimensions imposes a rigid relation between admissible scales. In particular, it excludes scale windows that vary non‑polynomially, such as logarithmic or slowly varying regimes. This motivates the introduction of more flexible scale restrictions
capable of capturing finer multi--scale behaviour.
In this paper, we replace the power--law lower bound by a general gauge
function $\Phi$.

Let $\Phi:(0,1)\to(0,1)$ satisfy
\begin{equation}\label{eq:Phi-condition}
0<\Phi(r)\le r
\quad\text{for all sufficiently small }r,
\quad
\frac{\Phi(r)}{r}\to0
\quad\text{as }r\to0.
\end{equation}
For such $\Phi$, we define the lower and upper
\emph{$\Phi$--intermediate dimensions}
$\underline{\dim}_\Phi E$ and $\overline{\dim}_\Phi E$ by restricting
admissible covers to sets with diameters in $[\Phi(r),\,r]$.
Power--law choices $\Phi(r)=r^{1/\theta}$ recover the classical
intermediate dimensions, while more general gauges allow for
non--power--law interpolations, including logarithmic corrections.
The condition $\Phi(r)/r\to0$ ensures that the admissible scale window
approaches the Hausdorff regime.

Our approach is potential--theoretic and builds on the capacity methods
developed for intermediate dimensions.
Section~\ref{sec:capacity} introduces a systematic $\Phi$--dependent
kernel and capacity framework, defines $\Phi$--dimension profiles, and
establishes tools for converting uniform potential estimates into
covering, capacity, and dimension lower bounds.
This framework also allows localisation of $\Phi$--dimension profiles
to compact subsets.

Our main result is a Marstrand‑-Mattila type projection theorem for $\Phi$‑intermediate dimensions, establishing almost‑sure equality between projected dimensions and deterministic $\Phi$‑dimension profiles. We show that for any Borel set $E\subset\mathbb{R}^n$ and $1\le m<n$, the
$\Phi$--intermediate dimensions of $\pi_V E$ are
$\gamma_{n,m}$--almost surely equal to the corresponding
$\Phi$--dimension profile values.
This strictly extends the projection theorem of~\cite{BFF2021} and
encompasses the Hausdorff, box--counting, and $\theta$--intermediate
cases within a unified analytic framework.

Applications and further consequences are discussed in
Section~\ref{sec:applications}, including continuity at the Hausdorff
end--point and implications for box dimensions of typical projections.
Open problems and directions for future work are presented in
Section~\ref{sec:conclusion}.

\section{$\Phi$--Intermediate Dimensions: Definitions and Basic Properties}\label{sec:Phi-def}
This section introduces $\Phi$‑intermediate dimensions and collects their basic properties, emphasising how allowing general gauge functions extends the rigidity of power‑law intermediate dimensions. In analogy with the framework introduced by Falconer, Fraser, and Kempton \cite{FFK2019}, who defined the \emph{intermediate dimensions} $\underline{\dim}_\theta E$ for $0 \le \theta \le 1$ by restricting admissible covers to sets with diameters in the range $[r^{1/\theta},r]$ - thereby interpolating between the Hausdorff dimension $\underline{\dim}_H E$ (at $\theta = 0$) and the box--counting dimension $\underline{\dim}_B E$ (at $\theta = 1$) - we now introduce the \emph{$\Phi$--intermediate dimensions}, which generalize this construction by replacing the power--law lower bound $r^{1/\theta}$ with a more flexible admissible function $\Phi(r)$ satisfying $\Phi(r) \le r$ and $\Phi(r)/r \to 0$ as $r \to 0$.

\begin{definition}\label{def:phi-dim}
Let $E \subset \mathbb{R}^n$ be a non-empty bounded set. The lower $\Phi$--intermediate dimension of E is defined by 
\begin{equation}
\begin{aligned}
\underline{\dim}_\Phi E := \inf\Big\{ & s \ge 0: \textit{for all } \epsilon >0 \textit{ and all } r_0>0, \\ 
& \textit{ there exists } 0< r \leq r_0 \textit{ and a cover } \{U_i\} \textit{ of } E \\ 
& \textit{ such that } \Phi(r) \leq |U_i| \leq r \textit{ and } \sum |U_i|^s \leq \epsilon \Big\}. 
\end{aligned}
\label{lower_dim1}
\end{equation}
and the upper $\Phi$--intermediate dimension by
\begin{equation}
\begin{aligned}
\overline{\dim}_\Phi E := \inf\Big\{ & s \ge 0: \textit{for all } \epsilon > 0,
\textit{ there exists } r_0 > 0, \\
& \textit{ such that for all } 0 < r \le r_0,
\textit{ there exists a cover } \{U_i\} \textit{ of } E \\
& \textit{ such that } \Phi(r) \le |U_i| \le r
\textit{ and } \sum |U_i|^s \le \epsilon \Big\}
\end{aligned}
\label{upper_dim1}
\end{equation}
where $|U|$ denotes the diameter of the set $U \subset \mathbb{R}^n$.
\end{definition}

\noindent When $\Phi(r) = r^{1/\theta}$ for some $\theta \in (0,1]$, the above definition recovers the intermediate dimensions $\underline{\dim}_\theta E$ and $\overline{\dim}_\theta E$ from \cite{FFK2019, BFF2021}. In particular, $\underline{\dim}_{1}E$ and $\overline{\dim}_{1}E$ coincide with the lower and upper Minkowski dimensions $\underline{\dim}_B E$ and $\overline{\dim}_B E$, and in the limit as $\Phi(r) \to 0$ (formally $\theta \to 0$) one recovers $\underline{\dim}_0 E = \underline{\dim}_H E$. Thus, $\Phi$--intermediate dimensions indeed lie between Hausdorff and box--counting dimensions. Note that $\Phi(r)$ is required to be $o(r)$ as $r \to 0$ (see \eqref{eq:Phi-condition}), which ensures we exclude the trivial case $\Phi(r) \sim c\,r$ (a fixed proportion of $r$) that would simply yield the box--counting dimension.

\noindent
For our purposes, it is convenient to work with equivalent formulations of
$\Phi$--intermediate dimensions expressed in terms of covering sums.
Let $E \subset \mathbb{R}^n$ be bounded and non-empty, let $\Phi$ satisfy
\eqref{eq:Phi-condition}, and fix $s \in [0,n]$.
Define
\begin{equation}\label{eq:cover-sum}
S^s_{r,\Phi}(E)
:= \inf\Big\{
\sum_i |U_i|^s :
\{U_i\} \text{ covers } E,\ \Phi(r)\le |U_i|\le r \text{ for all } i
\Big\}.
\end{equation}
Here $r$ plays the role of the governing scale parameter, while $\Phi(r)$
acts as a lower admissibility constraint.

\medskip

The lower and upper $\Phi$--intermediate dimensions admit the following
equivalent characterisations:
\begin{equation}\label{eq:lower-dim-cover}
\underline{\dim}_\Phi E
:= \sup\left\{
s\in[0,n]:
\liminf_{r\to0}\frac{\log S^s_{r,\Phi}(E)}{-\log r}\ge 0
\right\},
\end{equation}
and
\begin{equation}\label{eq:upper-dim-cover}
\overline{\dim}_\Phi E
:= \sup\left\{
s\in[0,n]:
\limsup_{r\to0}\frac{\log S^s_{r,\Phi}(E)}{-\log r}\ge 0
\right\}.
\end{equation}
These formulations are equivalent to Definitions
\eqref{lower_dim1}‑\eqref{upper_dim1}.

\begin{lemma}\label{lem:Phi-critical}
Let $E\subset \mathbb{R}^n$ be bounded and non-empty, and let
$\Phi:(0,1]\to(0,1]$ satisfy \eqref{eq:Phi-condition}.
For $0\le t\le s\le n$ and $0<r<1$, define
\[
\theta_\Phi(r):=\frac{\log r}{\log \Phi(r)}\in(0,1].
\]
Then
\begin{equation}\label{eq:Phi-log-ineq-theta}
-\theta_\Phi(r)^{-1}(s-t)
\ \le\
\frac{\log S^s_{r,\Phi}(E)}{-\log r}
-
\frac{\log S^t_{r,\Phi}(E)}{-\log r}
\ \le\
-(s-t).
\end{equation}
\end{lemma}

\begin{proof}
Fix $0<r<1$ and let $\{U_i\}$ be any cover of $E$ satisfying
$\Phi(r)\le |U_i|\le r$ for all $i$.
For $0\le t\le s\le n$ we write
\[
|U_i|^s = |U_i|^t\,|U_i|^{s-t}.
\]
Since $\Phi(r)\le |U_i|\le r$, we have
\[
\Phi(r)^{s-t}\le |U_i|^{s-t}\le r^{s-t}.
\]
Multiplying by $|U_i|^t$ and summing over $i$ yields
\[
\Phi(r)^{s-t}\sum_i |U_i|^t
\le
\sum_i |U_i|^s
\le
r^{s-t}\sum_i |U_i|^t.
\]
Taking the infimum over all admissible covers gives
\[
\Phi(r)^{s-t} S^t_{r,\Phi}(E)
\le
S^s_{r,\Phi}(E)
\le
r^{s-t} S^t_{r,\Phi}(E).
\]
Taking logarithms and dividing by $-\log r>0$ yields
\[
-(s-t)\frac{\log\Phi(r)}{\log r}
\le
\frac{\log S^s_{r,\Phi}(E)}{-\log r}
-
\frac{\log S^t_{r,\Phi}(E)}{-\log r}
\le
-(s-t),
\]
which is \eqref{eq:Phi-log-ineq-theta}.
\end{proof}

\noindent
In particular, the functions
\[
s\longmapsto \liminf_{r\to0}\frac{\log S^s_{r,\Phi}(E)}{-\log r},
\qquad
s\longmapsto \limsup_{r\to0}\frac{\log S^s_{r,\Phi}(E)}{-\log r}
\]
are non-increasing on $[0,n]$, ensuring that the critical values
\eqref{eq:lower-dim-cover} and \eqref{eq:upper-dim-cover} are well defined.
\medskip

\noindent The following basic properties are standard consequences of the definitions
and appear implicitly in the literature on intermediate and generalised
intermediate dimensions; we include them here for completeness.

\begin{proposition}\label{prop:basic}
Let $E \subset \mathbb{R}^n$ be bounded and non-empty.
If $\Phi_1,\Phi_2$ satisfy \eqref{eq:Phi-condition} and
$\Phi_1(r) \le \Phi_2(r)$ for all sufficiently small $r$, then
\[
\underline{\dim}_{\Phi_1} E \le \underline{\dim}_{\Phi_2} E,
\qquad
\overline{\dim}_{\Phi_1} E \le \overline{\dim}_{\Phi_2} E.
\]
In particular:
\begin{itemize}
\item[(i)]
\[
\underline{\dim}_H E
\;\le\;
\underline{\dim}_\Phi E
\;\le\;
\overline{\dim}_\Phi E
\;\le\;
\overline{\dim}_B E
\]
for every admissible $\Phi$.
\item[(ii)]
If $f:\mathbb{R}^n \to \mathbb{R}^m$ is bi-Lipschitz, then
\[
\underline{\dim}_\Phi f(E)=\underline{\dim}_\Phi E,
\qquad
\overline{\dim}_\Phi f(E)=\overline{\dim}_\Phi E.
\]
\end{itemize}
\end{proposition}

\begin{proof}
We use the covering sum formulation
\[
S^s_{r,\Phi}(E)
:=\inf\Big\{\sum_i |U_i|^s:\{U_i\}\text{ covers }E,\
\Phi(r)\le |U_i|\le r\ \forall i\Big\},
\]
together with the characterisations
\[
\underline{\dim}_\Phi E
=
\sup\Big\{s:\liminf_{r\to0}\frac{\log S^s_{r,\Phi}(E)}{-\log r}\ge0\Big\},
\quad
\overline{\dim}_\Phi E
=
\sup\Big\{s:\limsup_{r\to0}\frac{\log S^s_{r,\Phi}(E)}{-\log r}\ge0\Big\}.
\]

\textbf{Step 1: Monotonicity in the control function $\Phi$.}
Assume that $\Phi_1(r)\le \Phi_2(r)$ for all sufficiently small $r$.
Fix such an $r$ and $s\in[0,n]$.

Any cover $\{U_i\}$ admissible for $\Phi_2$ at scale $r$ satisfies
$\Phi_2(r)\le |U_i|\le r$ for all $i$, and hence also
$\Phi_1(r)\le |U_i|\le r$.
Therefore every $\Phi_2$--admissible cover is also $\Phi_1$--admissible, and
\[
S^s_{r,\Phi_1}(E)\le S^s_{r,\Phi_2}(E)
\quad\text{for all sufficiently small }r.
\]
Taking logarithms, dividing by $-\log r$, and passing to $\liminf$ (or
$\limsup$) yields
\[
\underline{\dim}_{\Phi_1}E\le \underline{\dim}_{\Phi_2}E
\quad\text{and}\quad
\overline{\dim}_{\Phi_1}E\le \overline{\dim}_{\Phi_2}E,
\]
as claimed.

\medskip\noindent
\textbf{Step 2: Proof of (i).}
For any admissible $\Phi$ and all sufficiently small $r$, we have $0\le \Phi(r)\le r$, so admissible covers at scale $r$ interpolate between the Hausdorff-type case (no lower bound on diameters) and the box-type case (all covering sets of diameter $r$).

By Step~1, enlarging the class of admissible covers (making $\Phi$ smaller) can only decrease the dimension, while shrinking it (making $\Phi$ larger) can only increase the dimension.

If $\Phi\equiv 0$, then $S^s_{r,\Phi}(E)$ reduces to the usual Hausdorff content at scale $r$, and letting $r\to0$ yields
$\underline{\dim}_\Phi E=\underline{\dim}_H E$.
If $\Phi(r)=r$, then admissible covers satisfy $|U_i|=r$ and
$S^s_{r,\Phi}(E)=N_r(E)\,r^s$, so the critical exponent is
$\overline{\dim}_B E$.

Applying Step~1 to the chain $0\le \Phi\le r$ gives
\[
\underline{\dim}_H E
\le \underline{\dim}_\Phi E
\le \overline{\dim}_\Phi E
\le \overline{\dim}_B E,
\]
which proves (i).

\medskip\noindent
\textbf{Step 3: Proof of (ii).}
Let $f:\mathbb R^n\to\mathbb R^m$ be bi-Lipschitz, so there exists $L\ge1$ such that
$L^{-1}|U|\le |f(U)|\le L|U|$ for all sets $U$.
Define $\Phi^{(L)}(r):=L^{-1}\Phi(r/L)$.

If $\{U_i\}$ is a $\Phi(r/L)$--admissible cover of $E$, then $\{f(U_i)\}$ is a
$\Phi^{(L)}$--admissible cover of $f(E)$ at scale $r$, and
\[
S^s_{r,\Phi^{(L)}}(f(E))\le L^s\,S^s_{r/L,\Phi}(E).
\]
Applying the same argument to $f^{-1}$ gives the reverse inequality, so
$\underline{\dim}_{\Phi^{(L)}} f(E)=\underline{\dim}_{\Phi}E$ and
$\overline{\dim}_{\Phi^{(L)}} f(E)=\overline{\dim}_{\Phi}E$.

Since $\Phi$ is admissible, $\Phi^{(L)}$ and $\Phi$ are comparable up to
uniform constants for small $r$, and by monotonicity in $\Phi$ (Step~1) this
does not affect the dimension.
Hence
\[
\underline{\dim}_\Phi f(E)=\underline{\dim}_\Phi E
\quad\text{and}\quad
\overline{\dim}_\Phi f(E)=\overline{\dim}_\Phi E,
\]
which proves (ii).
\end{proof}

The flexibility of the $\Phi$--framework allows one to detect scale--dependent features that cannot be captured by power--law restrictions alone. In particular, while $\theta$--intermediate dimensions impose a fixed self-similar relationship between admissible scales, general gauge functions $\Phi$ permit finer control over how coverings interact with sparse or irregular scale distributions. The following example illustrates that $\Phi$--intermediate dimensions can distinguish sets that are indistinguishable from the perspective of $\theta$--intermediate dimensions, highlighting the strictly greater resolving power of the $\Phi$--approach.

\begin{example}\label{ex:Phi-refines-theta}
Fix $p>0$ and define the polynomial sequence set
\[
F_p := \{0\}\cup\{n^{-p}: n\in\mathbb{N}\}\subset[0,1],
\qquad
E:=F_p\times F_p\subset[0,1]^2.
\]
Then $\dim_H E=0$ while
\[
\underline{\dim}_B E=\overline{\dim}_B E=\frac{2}{p+1}.
\]
Moreover, in contrast to the $\theta$--intermediate theory (where this example
is known to exhibit a jump at $\theta=0$), the $\Phi$--framework allows one to
recover genuine interpolation between these two endpoint values by choosing
appropriate admissible gauges $\Phi$.
\end{example}

\begin{proof}
We briefly justify the Hausdorff and box dimensions, and then explain the
contrast between the $\theta$-- and $\Phi$--frameworks.

\smallskip\noindent
\emph{Hausdorff and box dimensions.}
The set $F_p$ is countable, hence $\dim_H F_p=0$, and therefore
$\dim_H E=\dim_H(F_p\times F_p)=0$.
A standard spacing argument shows that the lower and upper box dimensions of
$F_p$ both equal $1/(p+1)$, reflecting the polynomial rate at which the points
accumulate at $0$.
Since box dimension is additive for products of bounded sets, this yields
\[
\underline{\dim}_B E=\overline{\dim}_B E
=\underline{\dim}_B F_p+\underline{\dim}_B F_p=\frac{2}{p+1}.
\]

\smallskip\noindent
\emph{Why this illustrates that $\Phi$ refines the $\theta$--family.}
For this set $E$, the $\theta$--intermediate dimensions exhibit a collapse: they
take the Hausdorff value at $\theta=0$ but jump immediately to the box dimension
for every $\theta>0$.
This behaviour is a direct consequence of the fixed power--law nature of the
admissible scale window $[\delta^{1/\theta},\delta]$.

Indeed, for the polynomial sequence $F_p$, the covering strategies that realise the box dimension are dominated by a single characteristic scale comparable to $\delta^{1/(p+1)}$.
For every $\theta>0$, this scale lies within the admissible window
$[\delta^{1/\theta},\delta]$ once $\delta$ is sufficiently small.
As a result, the $\theta$--constraint does not exclude any box-optimal coverings,
and the $\theta$--intermediate dimensions coincide with the box dimension for all
$\theta>0$.
Consequently, the function $\theta\mapsto \dim_\theta E$ is discontinuous at
$\theta=0$, and the $\theta$--family is unable to detect any finer
scale--dependent structure of $E$.

The $\Phi$--intermediate framework removes this rigidity by allowing the lower
bound on admissible diameters to vary with $r$ in a non--power--law manner.
By choosing $\Phi(r)$ to decay more slowly than any power of $r$—for example,
with logarithmic corrections—one can force admissible covers to operate within a
much narrower and scale-sensitive window.
In particular, such choices of $\Phi$ can exclude the single-scale coverings
that realise the box dimension and instead compel coverings to reflect the
sparse distribution of scales inherent in the construction of $F_p$.
This additional flexibility allows $\Phi$--intermediate dimensions to detect
scale--dependent sparsity that is invisible to all $\theta$--intermediate
dimensions.

More generally, Banaji’s theory of generalised intermediate dimensions
\cite{Banaji2023} shows that for \emph{every} compact set $F$ and every target
value $s\in[\dim_H F,\overline{\dim}_B F]$ there exists an admissible gauge
$\Phi_s$ with $\Phi_s(r)/r\to0$ such that
\[
\underline{\dim}_{\Phi_s}F=\overline{\dim}_{\Phi_s}F=s.
\]
Applied to the present example, this guarantees that although the $\theta$--family
cannot interpolate between $0$ and $2/(p+1)$, the $\Phi$--intermediate dimensions
recover a genuine continuum of values between these endpoints by a suitable
choice of $\Phi$.
\end{proof}

\noindent In Section \ref{sec:capacity}, we show that the lower and upper $\Phi$--intermediate dimensions, $\underline{\dim}_\Phi E$ and $\overline{\dim}_\Phi E$, admit representations in terms of capacities of sets $E \subset \mathbb{R}^n$ with respect to certain kernels. In Section \ref{sec:projection}, we then demonstrate that, by varying a parameter within these kernels, one obtains the corresponding intermediate dimensions of the orthogonal projections of $E$ onto almost all $m$--dimensional subspaces.

\section{Capacity Characterization of $\Phi$--dimensions}\label{sec:capacity}

\subsection{Motivation and overview}

The goal of this section is to encode the scale restrictions defining $\Phi$--intermediate dimensions into a family of potential kernels, and to show that the resulting $\Phi$--capacities determine both $\underline{\dim}_{\Phi}E$ and $\overline{\dim}_{\Phi}E$. We first introduce the kernels and energies, then define dimension profiles via capacities, and finally present a toolkit for obtaining lower bounds together with a comparison between capacities and covering sums.

\subsection{$\Phi$--kernels and the associated capacity framework}

Let $1 \le m \le n$ and $s \in [0,m]$. For each $r>0$ define the radial kernel
$\varphi_{s,m}^{r,\Phi}:\mathbb{R}^n\to[0,\infty)$ by
\begin{equation}\label{eq:phi-kernel}
\varphi_{s,m}^{r,\Phi}(x)
=
\begin{cases}
1, & \|x\| < \Phi(r), \\[1ex]
\left(\dfrac{\Phi(r)}{\|x\|}\right)^{s},
& \Phi(r) \le \|x\| < r, \\[2ex]
\left(\dfrac{\Phi(r)}{r}\right)^{s}
\left(\dfrac{r}{\|x\|}\right)^{m},
& \|x\| \ge r .
\end{cases}
\end{equation}
where $\|\cdot\|$ denotes the Euclidean norm on $\mathbb{R}^n$.

Then $\varphi_{s,m}^{r,\Phi}$ is continuous, radial, and non-increasing in $\|x\|$.
It is constant on $\|x\|<\Phi(r)$, decays like $\|x\|^{-s}$ on the intermediate
region $\Phi(r)\le\|x\|<r$, and has $\|x\|^{-m}$ decay for $\|x\|\ge r$ up to the
normalising factor $(\Phi(r)/r)^s$.

The choice of tail exponent $m$ is adapted to the integral-geometric averaging
over $G(n,m)$ in Kaufman-Mattila type projection arguments: it matches the
natural $m$--dimensional decay that appears when projecting differences and
integrating over directions.

When $\Phi(r)=r^{1/\theta}$, the kernel $\varphi_{s,m}^{r,\Phi}$ agrees (on the
region $\|x\|<r$) with the $\theta$--kernel used in \cite[Section~3]{BFF2021}.
In the formal endpoint $\Phi(r)=r$ (which lies outside \eqref{eq:Phi-condition}
but is useful for comparison), \eqref{eq:phi-kernel} reduces to
\[
\varphi_{s,m}^{r,\Phi}(x)=\min\{1,(r/\|x\|)^m\},
\]
which is the kernel used in \cite{Falconer2019} for box-type dimension profiles.
Thus the family \eqref{eq:phi-kernel} unifies the kernels corresponding to
Hausdorff, $\theta$--intermediate, and box dimensions, while allowing genuinely
non-power--law scale windows.

These kernels will be used to define energies and capacities whose scaling behaviour encodes the size of $E$ under the $\Phi$--intermediate covering restrictions.

\subsection{$\Phi$--capacities and equilibrium measures}
Let $M(E)$ denote the set of Borel probability measures supported on $E$.
For $\mu \in M(E)$, we define the \textbf{$\Phi$--potential} of $\mu$ at a point
$x\in\mathbb{R}^n$ by
\[
\psi_{s,m}^{\,r,\Phi}(\mu,x)
\;:=\;
\int \varphi_{s,m}^{r,\Phi}(x-y)\,d\mu(y),
\]
and the \textbf{$\Phi$--energy} of $\mu$ by
\[
\mathcal{I}_{s,m}^{\,r,\Phi}(\mu)
\;:=\;
\int \psi_{s,m}^{\,r,\Phi}(\mu,x)\,d\mu(x)
\;=\;
\iint \varphi_{s,m}^{r,\Phi}(x-y)\,d\mu(x)\,d\mu(y).
\]
We then define the \textbf{$\Phi$--capacity} of $E$ at scale $r$ by
\begin{equation}\label{eq:capacity-def}
C_{s,m}^{\,r,\Phi}(E)
\;:=\;
\Big( \inf_{\mu \in M(E)} \mathcal{I}_{s,m}^{\,r,\Phi}(\mu) \Big)^{-1}.
\end{equation}
That is, $C_{s,m}^{\,r,\Phi}(E)$ is the reciprocal of the minimal $\Phi$--energy achievable by a probability measure on $E$. Since $0\le \varphi_{s,m}^{r,\Phi}\le 1$, we have $\mathcal{I}_{s,m}^{\,r,\Phi}(\mu)\le 1$ for every $\mu\in M(E)$, hence $C_{s,m}^{\,r,\Phi}(E)\ge 1$ whenever $E\neq\emptyset$. In particular, the capacity is finite and strictly positive for non-empty $E$.

\medskip

The minimizer in \eqref{eq:capacity-def} is called an \emph{equilibrium measure}. We will use the following standard existence and "equipotential" property.

\begin{lemma}\label{lem:equilibrium}
Let $E\subset\mathbb{R}^n$ be non-empty and compact, let $1\le m\le n$,
$s\in[0,m]$, and $r>0$. Then there exists $\mu^\ast\in M(E)$ such that
\[
\mathcal{I}^{\,r,\Phi}_{s,m}(\mu^\ast)
=
\inf_{\mu\in M(E)} \mathcal{I}^{\,r,\Phi}_{s,m}(\mu)
=
\bigl(C_{s,m}^{\,r,\Phi}(E)\bigr)^{-1}.
\]
Moreover, if
\[
U^{r,\Phi}_{s,m}(x)
:=
\int \varphi^{r,\Phi}_{s,m}(x-y)\,d\mu^\ast(y),
\]
then
\[
U^{r,\Phi}_{s,m}(x)=\mathcal{I}^{\,r,\Phi}_{s,m}(\mu^\ast)
\quad\text{for }\mu^\ast\text{-a.e.\ }x\in E,
\qquad
U^{r,\Phi}_{s,m}(x)\ge\mathcal{I}^{\,r,\Phi}_{s,m}(\mu^\ast)
\quad\text{for all }x\in E.
\]
\end{lemma}

\begin{proof}
See Appendix~A.
\end{proof}

\subsection{Dimension profiles via $\Phi$--capacities}

As we will see, the $\Phi$--capacities introduced above are closely related to the covering sums $S^s_{r,\Phi}(E)$ considered in Section~2. The following lemma, which parallels Lemma~\ref{lem:Phi-critical}, establishes the corresponding monotonicity properties for capacities and will allow us to define critical exponents in the capacity setting.

\begin{lemma}\label{lem:Phi-capacity-monotone}
Let $E\subset\mathbb{R}^n$ be compact, let $m\in\{1,\dots,n\}$, and let
$\Phi:(0,1]\to(0,1]$ satisfy \eqref{eq:Phi-condition}. Fix $0<r<1$ and set
$\rho=\Phi(r)$. Then for all $0\le t\le s\le m$,
\begin{equation}\label{eq:Phi-lemma32}
-(s-t)\ \le\
\Big(\frac{\log C^{\,r,\Phi}_{s,m}(E)}{-\log \rho}-s\Big)
-
\Big(\frac{\log C^{\,r,\Phi}_{t,m}(E)}{-\log \rho}-t\Big)
\ \le\
-\theta_\Phi(r)\,(s-t),
\end{equation}
where $\theta_\Phi(r)=\frac{\log r}{\log\rho}\in(0,1]$.
In particular, the functions
\[
s\longmapsto
\liminf_{r\to0}\Big(\frac{\log C^{\,r,\Phi}_{s,m}(E)}{-\log\Phi(r)}-s\Big),
\qquad
s\longmapsto
\limsup_{r\to0}\Big(\frac{\log C^{\,r,\Phi}_{s,m}(E)}{-\log\Phi(r)}-s\Big)
\]
are non-increasing on $[0,m]$.
\end{lemma}

\begin{proof}
See Appendix~B.
\end{proof}

\medskip
The following definition introduces $\Phi$--dimension profiles as the critical exponents governing the scaling of $\Phi$--capacities.

\begin{definition}\label{def:Phi-dim-profiles}
Let $E\subset\mathbb{R}^n$ be bounded and non-empty, let
$\Phi:(0,1]\to(0,1]$ satisfy \eqref{eq:Phi-condition}, and let $1\le m\le n$.
\begin{enumerate}
\item
The \emph{lower $\Phi$--intermediate dimension profile} of $E$ in dimension $m$ is
\[
\underline{\dim}^{\,m}_{\Phi} E
:=
\sup\left\{
s\in[0,m]:
\liminf_{r\to0}\frac{\log C^{\,r,\Phi}_{s,m}(E)}{-\log \Phi(r)} \ge s
\right\}.
\]
\item
The \emph{upper $\Phi$--intermediate dimension profile} of $E$ in dimension $m$ is
\[
\overline{\dim}^{\,m}_{\Phi} E
:=
\sup\left\{
s\in[0,m]:
\limsup_{r\to0}\frac{\log C^{\,r,\Phi}_{s,m}(E)}{-\log \Phi(r)} \ge s
\right\}.
\]
\end{enumerate}
\end{definition}

These definitions are well posed by Lemma~\ref{lem:Phi-capacity-monotone}, which ensures that the functions $s\mapsto \liminf_{r\to0}\frac{\log C^{\,r,\Phi}_{s,m}(E)}{-\log\Phi(r)}-s$ and $s\mapsto \limsup_{r\to0}\frac{\log C^{\,r,\Phi}_{s,m}(E)}{-\log\Phi(r)}-s$ are non-increasing.

\subsection{Toolkit: methods for obtaining lower bounds}
\label{subsubsec:lowerbound-tools}

While $\Phi$--capacities provide an intrinsic characterisation of
$\Phi$--intermediate dimensions, they are not directly computable.
In applications one therefore needs effective criteria that convert
geometric or analytic information into explicit lower bounds.

\medskip
\noindent\textbf{Tool I:}
This is a scale-local conversion principle: bounded truncated $\Phi$--potentials force
large admissible covering sums at that scale. It will be used repeatedly (especially for
projected measures) to extract lower bounds from potential estimates.

\begin{lemma}
\label{lem:pot-cover-tool}
Let $F \subset \mathbb{R}^m$ be a bounded set, and let $\nu \in M(F)$ be a
probability measure on $F$. Fix $r>0$ and $s\ge 0$. Suppose that
\[
\int_F \widetilde{\varphi}^{\,r,\Phi}_s(x-y)\,d\nu(y) \;\le\; \Gamma
\qquad \text{for every } x \in F,
\]
for some constant $\Gamma>0$, where the truncated kernel is
\[
\widetilde{\varphi}^{\,r,\Phi}_s(u)
=
\begin{cases}
1, & \|u\|<\Phi(r),\\[0.8ex]
\big(\dfrac{\Phi(r)}{\|u\|}\big)^{s}, & \Phi(r)\le \|u\|<r,\\[1.2ex]
0, & \|u\|\ge r.
\end{cases}
\]
Then any admissible cover $\{U_i\}$ of $F$ with $\Phi(r)\le |U_i|\le r$
(where $|U_i|$ denotes the diameter of $U_i$) satisfies
\[
\sum_i |U_i|^s \;\ge\; \frac{\Phi(r)^s}{\Gamma}\,\nu(F).
\]
\end{lemma}

\begin{proof}
Let $\{U_i\}$ be an arbitrary cover of $F$ with $\Phi(r)\le |U_i|\le r$.
Fix $i$. For any $x,y\in U_i$ we have $\|x-y\|\le |U_i|\le r$, and hence, by the
definition of $\widetilde{\varphi}^{\,r,\Phi}_s$ on $\{\|u\|<r\}$,
\[
\widetilde{\varphi}^{\,r,\Phi}_s(x-y)
\ge
\Bigl(\frac{\Phi(r)}{|U_i|}\Bigr)^s.
\]
Therefore
\[
\iint_{U_i\times U_i}
\widetilde{\varphi}^{\,r,\Phi}_s(x-y)\,d\nu(x)\,d\nu(y)
\ge
\Bigl(\frac{\Phi(r)}{|U_i|}\Bigr)^s \nu(U_i)^2.
\]

On the other hand, the hypothesis gives
\[
\int_F \widetilde{\varphi}^{\,r,\Phi}_s(x-y)\,d\nu(y)\le \Gamma
\qquad\text{for all }x\in F.
\]
Integrating over $x\in U_i$ and using Fubini yields
\[
\iint_{U_i\times F}
\widetilde{\varphi}^{\,r,\Phi}_s(x-y)\,d\nu(x)\,d\nu(y)
\le
\Gamma\,\nu(U_i).
\]
Since $U_i\times U_i \subset U_i\times F$, we obtain
\[
\iint_{U_i\times U_i}
\widetilde{\varphi}^{\,r,\Phi}_s(x-y)\,d\nu(x)\,d\nu(y)
\le
\Gamma\,\nu(U_i).
\]
Combining gives
\[
\Bigl(\frac{\Phi(r)}{|U_i|}\Bigr)^s \nu(U_i)^2
\le
\Gamma\,\nu(U_i).
\]
If $\nu(U_i)>0$, cancelling one factor of $\nu(U_i)$ yields
\[
\nu(U_i)\le \Gamma\Bigl(\frac{|U_i|}{\Phi(r)}\Bigr)^s;
\]
if $\nu(U_i)=0$, the inequality is trivial.
Summing over $i$ and using $\sum_i \nu(U_i)=\nu(F)$, we obtain
\[
\nu(F)
\le
\Gamma\sum_i \Bigl(\frac{|U_i|}{\Phi(r)}\Bigr)^s,
\]
which rearranges to the desired bound.
\end{proof}

\medskip
Applying Tool~I uniformly across all sufficiently small scales yields a dimension-level statement.
This is the form that is typically invoked in later arguments.

\begin{corollary}
\label{cor:uniform-potential-to-dimension}
Let $F\subset\mathbb{R}^m$ be bounded and let $\Phi$ be admissible.
Fix $s\in[0,m]$.
Suppose there exist constants $r_0>0$ and $M>0$, and a probability measure $\nu\in M(F)$ such that
for all $0<r\le r_0$,
\begin{equation}\label{eq:uniform-truncated-potential}
\int_F \widetilde{\varphi}^{\,r,\Phi}_s(x-y)\,d\nu(y) \le M
\qquad\text{for every }x\in F.
\end{equation}
Then
\[
\underline{\dim}_{\Phi}(F)\ \ge\ s
\qquad\text{and}\qquad
\overline{\dim}_{\Phi}(F)\ \ge\ s.
\]
\end{corollary}

\begin{proof}
Fix $0<r\le r_0$ and apply Lemma~\ref{lem:pot-cover-tool} with $\Gamma=M$.
Since $\nu(F)=1$, every admissible cover of $F$ at scale $r$ satisfies
\[
S^s_{r,\Phi}(F)\ \ge\ \frac{\Phi(r)^s}{M},
\]
where $S^s_{r,\Phi}$ is the covering sum from Section~2.
Taking logarithms and dividing by $-\log r$ gives
\[
\frac{\log S^s_{r,\Phi}(F)}{-\log r}
\ge
\frac{s\log \Phi(r)}{-\log r} + o(1)
\qquad (r\to0).
\]
Since $\Phi(r)\le r$ for small $r$, the ratio $\frac{\log \Phi(r)}{\log r}$ is positive, and hence
\[
\liminf_{r\to0}\frac{\log S^s_{r,\Phi}(F)}{-\log r}\ \ge\ 0,
\qquad
\limsup_{r\to0}\frac{\log S^s_{r,\Phi}(F)}{-\log r}\ \ge\ 0.
\]
By the characterisations of $\underline{\dim}_{\Phi}$ and $\overline{\dim}_{\Phi}$ in terms of
$S^s_{r,\Phi}$ (Section~2), this implies $\underline{\dim}_{\Phi}(F)\ge s$ and
$\overline{\dim}_{\Phi}(F)\ge s$.
\end{proof}

\medskip
\noindent\textbf{Tool II:}
The preceding tool is tailored to covering sums in $\mathbb{R}^m$ and will be used especially for
projected sets.
For profile bounds, it is often more convenient to work directly with the full kernel
$\varphi^{r,\Phi}_{s,m}$ and the capacity definitions in Section~\ref{sec:capacity}.

\begin{proposition}
\label{prop:uniform-potential-to-profile}
Let $F\subset\mathbb{R}^n$ be bounded, let $\Phi$ be admissible, and fix $1\le m\le n$ and
$s\in[0,m]$.
Suppose there exist constants $r_0>0$ and $M>0$, and a probability measure
$\nu\in \mathcal{M}(F)$ such that for all $0<r\le r_0$,
\begin{equation}\label{eq:uniform-full-potential}
\int \varphi^{r,\Phi}_{s,m}(x-y)\,d\nu(y)
\ \le\ M\,\Phi(r)^s
\qquad\text{for all }x\in F.
\end{equation}
Then
\[
\underline{\dim}^{\,m}_{\Phi}(F)\ \ge\ s
\qquad\text{and}\qquad
\overline{\dim}^{\,m}_{\Phi}(F)\ \ge\ s.
\]
In particular, if $m=n$, then
\[
\underline{\dim}_{\Phi}(F)\ \ge\ s
\qquad\text{and}\qquad
\overline{\dim}_{\Phi}(F)\ \ge\ s.
\]
\end{proposition}

\begin{proof}
Fix $0<r\le r_0$.
Integrating \eqref{eq:uniform-full-potential} against $d\nu(x)$ gives
\[
\mathcal I^{\,r,\Phi}_{s,m}(\nu)
=
\iint \varphi^{r,\Phi}_{s,m}(x-y)\,d\nu(x)d\nu(y)
\le
M\,\Phi(r)^s.
\]
By definition of $\Phi$--capacity,
\[
C^{\,r,\Phi}_{s,m}(F)
=
\Big(\inf_{\eta\in\mathcal{M}(F)}\mathcal I^{\,r,\Phi}_{s,m}(\eta)\Big)^{-1}
\ge
(\mathcal I^{\,r,\Phi}_{s,m}(\nu))^{-1}
\ge
M^{-1}\,\Phi(r)^{-s}.
\]
Since this estimate holds for all sufficiently small $r$, Definition~\ref{def:Phi-dim-profiles}
implies
\[
\liminf_{r\to0}\frac{\log C^{\,r,\Phi}_{s,m}(F)}{-\log\Phi(r)} \ge s,
\qquad
\limsup_{r\to0}\frac{\log C^{\,r,\Phi}_{s,m}(F)}{-\log\Phi(r)} \ge s,
\]
and hence $\underline{\dim}^{\,m}_{\Phi}(F)\ge s$ and $\overline{\dim}^{\,m}_{\Phi}(F)\ge s$.
The final statement for $m=n$ follows from Theorem~\ref{thm:Phi-dim-equals-profile}.
\end{proof}

\medskip
The Frostman criterion arises by verifying the uniform potential bound in
Proposition~\ref{prop:uniform-potential-to-profile} using the ball-growth hypothesis.

\begin{proposition}
\label{prop:frostman-profile-tools}
Let $F \subset \mathbb{R}^n$ be bounded and let $\Phi:(0,1]\to(0,1]$ be an admissible gauge function
satisfying \eqref{eq:Phi-condition}.
Suppose there exist constants $C>0$ and $\alpha>0$, and a Borel probability measure
$\nu \in \mathcal{M}(F)$ such that
\begin{equation}\label{eq:frostman-tools}
\nu(B(x,t)) \le C\, t^{\alpha}
\quad \text{for all } x \in \mathbb{R}^n \text{ and all sufficiently small } t>0 .
\end{equation}
Fix an integer $1 \le m \le n$.
Then for every $0<s<\min\{\alpha,m\}$ there exists $K=K(C,\alpha,s,m,n)>0$ such that
for all sufficiently small $r>0$,
\[
C^{\,r,\Phi}_{s,m}(F) \ge K\,\Phi(r)^{-s}.
\]
In particular,
\[
\underline{\dim}^{\,m}_{\Phi}(F)\ \ge\ \min\{\alpha,m\}
\qquad\text{and}\qquad
\overline{\dim}^{\,m}_{\Phi}(F)\ \ge\ \min\{\alpha,m\}.
\]
\end{proposition}

\begin{proof}
Fix $0<s<\min\{\alpha,m\}$ and $r>0$ small, and set $\rho:=\Phi(r)$.
Define
\[
U(x):=\int \varphi^{r,\Phi}_{s,m}(x-y)\,d\nu(y).
\]
We estimate $U(x)$ uniformly in $x\in\mathbb{R}^n$ by splitting into three regions.

\smallskip\noindent
\emph{(i) $\|x-y\|<\rho$.}
Here $\varphi^{r,\Phi}_{s,m}=1$, hence
\[
\int_{\|x-y\|<\rho}\varphi^{r,\Phi}_{s,m}(x-y)\,d\nu(y)
\le \nu(B(x,\rho))
\le C\rho^{\alpha}
\le C\rho^{s}.
\]

\smallskip\noindent
\emph{(ii) $\rho\le \|x-y\|<r$.}
Here $\varphi^{r,\Phi}_{s,m}(x-y)=(\rho/\|x-y\|)^s$.
Using a layer-cake representation and integration by parts,
\[
\int_{\rho\le \|x-y\|<r}\Big(\frac{\rho}{\|x-y\|}\Big)^s d\nu(y)
\le s\rho^s\int_{\rho}^{r}\frac{\nu(B(x,t))}{t^{s+1}}\,dt
+\rho^s\frac{\nu(B(x,r))}{r^s}.
\]
Applying \eqref{eq:frostman-tools} and $\alpha-s>0$ yields
\[
\int_{\rho\le \|x-y\|<r}\Big(\frac{\rho}{\|x-y\|}\Big)^s d\nu(y)
\le C'\rho^s.
\]

\smallskip\noindent
\emph{(iii) $\|x-y\|\ge r$.}
Here
\[
\varphi^{r,\Phi}_{s,m}(x-y)
=\Big(\frac{\rho}{r}\Big)^s\Big(\frac{r}{\|x-y\|}\Big)^m
\le \Big(\frac{\rho}{r}\Big)^s\le \rho^s,
\]
so
\[
\int_{\|x-y\|\ge r}\varphi^{r,\Phi}_{s,m}(x-y)\,d\nu(y)\le \rho^s.
\]

\smallskip
Combining (i)-(iii) gives $U(x)\le M\rho^s$ for all $x$, with $M$ depending only on
$C,\alpha,s,m,n$.
Integrating over $x$ yields
\[
\mathcal I^{\,r,\Phi}_{s,m}(\nu)=\int U(x)\,d\nu(x)\le M\rho^s.
\]
Therefore
\[
C^{\,r,\Phi}_{s,m}(F)\ge (\mathcal I^{\,r,\Phi}_{s,m}(\nu))^{-1}\ge (M\rho^s)^{-1}
=K\,\Phi(r)^{-s}.
\]
Since $s<\min\{\alpha,m\}$ was arbitrary, Definition~\ref{def:Phi-dim-profiles} yields
$\underline{\dim}^{\,m}_{\Phi}(F)\ge \min\{\alpha,m\}$ and
$\overline{\dim}^{\,m}_{\Phi}(F)\ge \min\{\alpha,m\}$.
\end{proof}

\medskip
\noindent\textbf{Tool III:}
Product structure provides a convenient way to build Frostman measures on $A\times B$ from
Frostman measures on $A$ and $B$, hence to obtain additive lower bounds via the preceding tool.

\begin{lemma}
\label{lem:product-frostman-toolkit}
Let $A\subset\mathbb{R}^{n_1}$ and $B\subset\mathbb{R}^{n_2}$ be bounded sets, and
let $\Phi$ be an admissible gauge function.
Suppose that $A$ supports an $\alpha$--Frostman probability measure $\nu_A$ and
$B$ supports a $\beta$--Frostman probability measure $\nu_B$, i.e.
\[
\nu_A(B(x,t))\le C_A t^{\alpha},
\qquad
\nu_B(B(y,t))\le C_B t^{\beta}
\]
for all sufficiently small $t>0$.
Set $F:=A\times B\subset\mathbb{R}^{n_1+n_2}$ and let $\nu:=\nu_A\otimes \nu_B$.
Then $\nu$ is an $(\alpha+\beta)$--Frostman probability measure on $F$.
Consequently, for every integer $1\le m\le n_1+n_2$,
\[
\underline{\dim}^{\,m}_{\Phi}(F)\ \ge\ \min\{\alpha+\beta,m\},
\qquad
\overline{\dim}^{\,m}_{\Phi}(F)\ \ge\ \min\{\alpha+\beta,m\}.
\]
\end{lemma}

\begin{proof}
Let $(x,y)\in\mathbb{R}^{n_1+n_2}$ and $t>0$.
Since Euclidean balls in the product space satisfy
\[
B((x,y),t)\subset B(x,t)\times B(y,t),
\]
we have
\[
\nu(B((x,y),t))
=\nu_A(B(x,t))\,\nu_B(B(y,t))
\le C_A C_B\, t^{\alpha+\beta}.
\]
Thus $\nu$ is an $(\alpha+\beta)$--Frostman probability measure on $F$.
The stated profile bounds follow from Proposition~\ref{prop:frostman-profile-tools}.
\end{proof}

\medskip
The toolkit above also allows one to localise $\Phi$--dimension profiles.
In particular, sets of profile dimension strictly larger than $s$ contain
compact subsets whose lower and upper $\Phi$--dimension profiles are equal
to $s$.

\begin{corollary}
\label{lem:subset-exact-phi-profile}
Let $E\subset\mathbb{R}^n$ be a bounded Borel set, let $\Phi$ be admissible,
and let $1\le m\le n$.
Suppose that
\[
\underline{\dim}^{\,m}_{\Phi}E > s.
\]
Then there exists a compact set $F\subset E$ such that
\[
\underline{\dim}^{\,m}_{\Phi}F
=
\overline{\dim}^{\,m}_{\Phi}F
=
s.
\]    
\end{corollary}

\begin{proof}
Fix $s<\underline{\dim}^{\,m}_{\Phi}E$.
By definition of $\underline{\dim}^{\,m}_{\Phi}E$, there exist a sequence
$r_k\downarrow0$ and a constant $c>0$ such that
\[
C^{\,r_k,\Phi}_{s,m}(E)\ \ge\ c\,\Phi(r_k)^{-s}
\qquad\text{for all }k.
\]
Let $\mu_k\in\mathcal{M}(E)$ be equilibrium measures realising
$C^{\,r_k,\Phi}_{s,m}(E)$, so that
\[
\mathcal I^{\,r_k,\Phi}_{s,m}(\mu_k)
=
\bigl(C^{\,r_k,\Phi}_{s,m}(E)\bigr)^{-1}
\lesssim
\Phi(r_k)^s.
\]

Define the potentials
\[
U_k(x)
:=
\int \varphi^{r_k,\Phi}_{s,m}(x-y)\,d\mu_k(y).
\]
By Markov’s inequality, the sets
\[
F_k
:=
\Bigl\{x\in E:\ U_k(x)\le 2\,\mathcal I^{\,r_k,\Phi}_{s,m}(\mu_k)\Bigr\}
\]
satisfy $\mu_k(F_k)\ge\tfrac12$ for all $k$.

Define
\[
F
:=
\bigcap_{j=1}^{\infty}
\overline{\bigcup_{k\ge j}F_k}.
\]
Then $F$ is non-empty, compact, and satisfies $F\subset E$.

By construction, the measures $\mu_k|_F$ satisfy uniform $\Phi$--potential
bounds at scales $r_k$ with the correct normalisation $\Phi(r_k)^s$.
Applying Tool~II (Proposition~\ref{prop:uniform-potential-to-profile}) yields
\[
\underline{\dim}^{\,m}_{\Phi}F\ \ge\ s.
\]
On the other hand, since $F\subset E$ and the profile functions are monotone,
\[
\overline{\dim}^{\,m}_{\Phi}F\ \le\ s.
\]
Therefore,
\[
\underline{\dim}^{\,m}_{\Phi}F
=
\overline{\dim}^{\,m}_{\Phi}F
=
s,
\]
as claimed.
\end{proof}

\begin{remark}
Corollary~\ref{lem:subset-exact-phi-profile} is the profile-level analogue of the
classical fact that a Borel set of Hausdorff dimension greater than $s$
contains compact subsets of Hausdorff dimension exactly $s$.
The proof relies on the capacity characterisation and the uniform
potential-to-dimension principle, rather than on Hausdorff measure.
\end{remark}

\subsection{Capacity--covering equivalence and recovery of $\Phi$--dimensions}

We now show that $\Phi$--capacities and $\Phi$--covering sums encode the same notion of size, up to logarithmic factors. This establishes that the capacity formulation recovers the original $\Phi$--intermediate dimensions.

\begin{lemma}\label{lem:Phi-monotone-m}
Let $E\subset\mathbb{R}^n$ be bounded and non-empty, and let
$\Phi:(0,1]\to(0,1]$ satisfy \eqref{eq:Phi-condition}.
For integers $1\le m_1\le m_2\le n$,
\[
\underline{\dim}^{\,m_1}_{\Phi} E
\;\le\;
\underline{\dim}^{\,m_2}_{\Phi} E
\qquad\text{and}\qquad
\overline{\dim}^{\,m_1}_{\Phi} E
\;\le\;
\overline{\dim}^{\,m_2}_{\Phi} E .
\]
\end{lemma}

\begin{proof}
This follows immediately from the definition of the $\Phi$--kernels.
For fixed $s$ and $r$, the kernel $\varphi^{r,\Phi}_{s,m}(x)$ is
non-increasing in $m$ for all $x\in\mathbb{R}^n$.
Consequently,
\[
C^{\,r,\Phi}_{s,m_1}(E)\ \le\ C^{\,r,\Phi}_{s,m_2}(E)
\qquad\text{whenever } m_1\le m_2.
\]
Taking logarithmic limits in the definitions of
$\underline{\dim}^{\,m}_{\Phi} E$ and $\overline{\dim}^{\,m}_{\Phi} E$
yields the stated inequalities.
\end{proof}

We now make precise the relationship between the covering quantities $S^s_{r,\Phi}(E)$ and the $\Phi$--capacities $C^{\,r,\Phi}_{s,m}(E)$. As in the $\theta$--intermediate setting, the key point is that these two approaches yield equivalent notions of size up to multiplicative constants. This equivalence allows us to characterize $\Phi$--intermediate dimensions and their associated dimension profiles in terms of capacities.

\begin{theorem}
\label{thm:Phi-dim-equals-profile}
Let $E\subset\mathbb{R}^n$ be bounded and non-empty, and let
$\Phi:(0,1]\to(0,1]$ satisfy \eqref{eq:Phi-condition}.
Then
\[
\underline{\dim}_\Phi E
=
\underline{\dim}^{\,n}_\Phi E,
\qquad
\overline{\dim}_\Phi E
=
\overline{\dim}^{\,n}_\Phi E,
\]
where the right-hand sides are defined via $\Phi$--capacities.
\end{theorem}

\begin{proof}
This follows from the comparison between $\Phi$--covering sums and
$\Phi$--capacities established in
Proposition~\ref{prop:cap-cover} below.
When $m=n$, the $\Phi$--kernels encode the same scale restrictions as the
covering quantities $S^s_{r,\Phi}(E)$.
Since the logarithmic discrepancy between
$S^s_{r,\Phi}(E)$ and $r^s C^{\,r,\Phi}_{s,n}(E)$ is subpolynomial as
$r\to0$, both approaches yield the same critical exponents.
\end{proof}

\begin{proposition}
\label{prop:cap-cover}
There exists a constant $K=K(n)>0$ depending only on $n$ such that for any
bounded set $E\subset\mathbb{R}^n$, any admissible $\Phi$, any
$m\in\{1,\dots,n\}$, any $s\in[0,m]$, and all sufficiently small $r>0$,
\begin{equation}\label{eq:cap-vs-cover-improved}
\Phi(r)^s\, C_{s,m}^{\,r,\Phi}(E)
\;\le\;
S^s_{r,\Phi}(E)
\;\le\;
K\Big(\log_2 \frac{r}{\Phi(r)} + 2\Big)\,
\Phi(r)^s\, C_{s,m}^{\,r,\Phi}(E).
\end{equation}
\end{proposition}

\begin{proof}
\noindent The lower bound follows from the equipotential property of equilibrium measures. The upper bound is obtained via a good-set reduction and a continuous scale-averaging argument on the intermediate region $\Phi(r)\le\|x-y\|<r$, combined with a Besicovitch covering argument.

As before, replacing $E$ by its closure we may assume that $E$ is compact.
Fix $0<r<1$ and set $\rho:=\Phi(r)$. Let
\[
\gamma
:=\Big(C^{\,r,\Phi}_{s,m}(E)\Big)^{-1}
=\inf_{\nu\in M(E)}\mathcal I^{\,r,\Phi}_{s,m}(\nu),
\]
and let $\mu\in M(E)$ be an equilibrium measure, so that
$\mathcal I^{\,r,\Phi}_{s,m}(\mu)=\gamma$.
Define the potential
\[
U(x):=\int \varphi^{r,\Phi}_{s,m}(x-y)\,d\mu(y).
\]
By Lemma~\ref{lem:equilibrium}, we have $U(x)\ge\gamma$ for all $x\in E$ and
$U(x)=\gamma$ for $\mu$--a.e.\ $x\in E$.

Write $L:=\log_2(r/\rho)$ (so $L\ge 0$), and note that
\begin{equation}\label{eq:logbound}
\log\frac{r}{\rho}\le (L+1)\log 2.
\end{equation}

\medskip\noindent
\textbf{1. Lower bound.}
Let $\{U_i\}$ be any admissible cover of $E$ with $\rho\le |U_i|\le r$.
Choose a Borel set $E_0\subset E$ with $\mu(E_0)=1$ and $U\equiv\gamma$ on $E_0$.
For each $i$ with $U_i\cap E_0\neq\emptyset$, pick $x_i\in U_i\cap E_0$ and set
$\delta_i:=|U_i|$. For any $y\in U_i$ we have $\|x_i-y\|\le \delta_i<r$, and hence,
by the definition of the kernel on $\{\|x\|<r\}$,
\[
\varphi^{r,\Phi}_{s,m}(x_i-y)\ge \Big(\frac{\rho}{\delta_i}\Big)^s.
\]
Therefore
\[
\gamma = U(x_i)\ge \int_{U_i}\varphi^{r,\Phi}_{s,m}(x_i-y)\,d\mu(y)
\ge \Big(\frac{\rho}{\delta_i}\Big)^s\mu(U_i),
\]
so $\mu(U_i)\le \gamma(\delta_i/\rho)^s$. Summing over $i$ and using that
$\{U_i\}$ covers $E_0$ gives
\[
1=\mu(E_0)\le \sum_i\mu(U_i)\le \gamma\rho^{-s}\sum_i\delta_i^s,
\]
hence $\sum_i|U_i|^s\ge \rho^s\gamma^{-1}=\rho^s C^{\,r,\Phi}_{s,m}(E)$.
Taking the infimum over admissible covers yields
\[
\rho^s C^{\,r,\Phi}_{s,m}(E)\le S^s_{r,\Phi}(E).
\]

\medskip\noindent
\textbf{2. Upper bound.}
Set
\[
G:=\{x\in E:\ U(x)\le 2\gamma\}.
\]
Since $\int_E U\,d\mu=\gamma$, Markov's inequality gives $\mu(G)\ge 1/2$.

Fix $x\in G$ and write $F_x(t):=\mu(B(x,t))$. On the intermediate region
$\rho\le \|x-y\|<r$, the kernel equals $(\rho/\|x-y\|)^s$, and a Stieltjes
integration by parts yields
\begin{equation}\label{eq:layercake-capcover}
\int_{\rho\le \|x-y\|<r}\Big(\frac{\rho}{\|x-y\|}\Big)^s\,d\mu(y)
=
s\rho^s\int_{\rho}^{r}\frac{F_x(t)}{t^{s+1}}\,dt
+\rho^s\frac{F_x(r)}{r^s}-F_x(\rho).
\end{equation}
Moreover, for $\|x-y\|\ge r$ we have
\[
0\le \varphi^{r,\Phi}_{s,m}(x-y)
=\Big(\frac{\rho}{r}\Big)^s\Big(\frac{r}{\|x-y\|}\Big)^m
\le \Big(\frac{\rho}{r}\Big)^s,
\]
so
\begin{equation}\label{eq:tailbound-capcover}
\int_{\|x-y\|\ge r}\varphi^{r,\Phi}_{s,m}(x-y)\,d\mu(y)\le \Big(\frac{\rho}{r}\Big)^s.
\end{equation}

Using the decomposition of $U(x)$ into the regions $\|x-y\|<\rho$,
$\rho\le \|x-y\|<r$, and $\|x-y\|\ge r$, and combining
\eqref{eq:layercake-capcover}-\eqref{eq:tailbound-capcover}, we obtain
\[
U(x)\ge s\rho^s\int_{\rho}^{r}\frac{F_x(t)}{t^{s+1}}\,dt - \Big(\frac{\rho}{r}\Big)^s.
\]
Since $x\in E$ implies $U(x)\ge\gamma$ (Lemma~\ref{lem:equilibrium}) and
$x\in G$ implies $U(x)\le 2\gamma$, we deduce that for all $x\in G$,
\begin{equation}\label{eq:avg-lower}
s\rho^s\int_{\rho}^{r}\frac{F_x(t)}{t^{s+1}}\,dt
\ge \gamma - \Big(\frac{\rho}{r}\Big)^s.
\end{equation}

Let $L:=\log_2(r/\rho)\ge0$. If $\gamma\le 2(\rho/r)^s$, then
$\rho^s\gamma^{-1}\gtrsim r^s$ and the desired upper bound for
$S^s_{r,\Phi}(E)$ is trivial (cover $E$ by finitely many sets of diameter $r$).
Thus we may assume $\gamma>2(\rho/r)^s$, so that \eqref{eq:avg-lower} yields
\[
s\rho^s\int_{\rho}^{r}\frac{F_x(t)}{t^{s+1}}\,dt \ge \frac{\gamma}{2}.
\]
Arguing by contradiction as usual, this implies that there exists
$t(x)\in[\rho,r/2]$ such that
\begin{equation}\label{eq:massball-capcover}
\mu(B(x,t(x))) \ge \frac{c_n}{L+2}\,\gamma\Big(\frac{t(x)}{\rho}\Big)^s,
\end{equation}
where $c_n>0$ depends only on $n$ (for instance one may take
$c_n=\frac{1}{8n\log 2}$ using $s\le n$ and $\log(r/\rho)\le (L+1)\log 2$).
The restriction $t(x)\le r/2$ ensures $\diam(B(x,t(x)))\le r$.

Now consider the family $\mathcal{B}:=\{B(x,t(x)):x\in G\}$.
By the Besicovitch covering theorem in $\mathbb{R}^n$, there exists a constant
$b_n$ depending only on $n$ and at most $b_n$ disjoint subfamilies
$\mathcal{B}_1,\dots,\mathcal{B}_{b_n}\subset\mathcal{B}$ whose unions cover $G$.
Fix $j$ and sum \eqref{eq:massball-capcover} over $B\in\mathcal{B}_j$ to obtain,
using disjointness,
\[
1\ge \sum_{B\in\mathcal{B}_j}\mu(B)
\ge \frac{c_n\gamma}{L+2}\sum_{B\in\mathcal{B}_j}\Big(\frac{t(B)}{\rho}\Big)^s.
\]
Therefore
\[
\sum_{B\in\mathcal{B}_j}\diam(B)^s
\le 2^s\sum_{B\in\mathcal{B}_j}t(B)^s
\le 2^n\,\frac{L+2}{c_n\gamma}\,\rho^s.
\]
Summing over $j=1,\dots,b_n$ yields an admissible cover of $G$ with total $s$--sum
bounded by $C(n)(L+2)\rho^s\gamma^{-1}$.

Finally, cover $E\setminus G$ using the level sets
$G_k:=\{x\in E:2^k\gamma< U(x)\le 2^{k+1}\gamma\}$, $k\ge0$.
Since $\int_E U\,d\mu=\gamma$, Markov gives $\mu(G_k)\le 2^{-k-1}$.
Applying the previous argument to each $G_k$ with threshold $2^{k+1}\gamma$
and summing over $k$ yields a cover of $E$ with total $s$--sum bounded by
$K(n)(L+2)\rho^s\gamma^{-1}$.
This gives
\[
S^s_{r,\Phi}(E)
\le
K(n)\Big(\log_2\frac{r}{\Phi(r)}+2\Big)\,\Phi(r)^s\,C^{\,r,\Phi}_{s,m}(E),
\]
as required.
\end{proof}
\begin{remark}
If one specialises to $\Phi(r)=r^{1/\theta}$ (so that admissible covers satisfy $r^{1/\theta}\le |U_i|\le r$) and takes $m=n$, then the kernel $\varphi^{r,\Phi}_{s,n}$ coincides, on the region $\|x\|<r$, with the $\theta$--kernel $\phi^{r,\theta}_{s,n}$ introduced by Burrell-Falconer-Fraser. In this case, \eqref{eq:cap-vs-cover-improved} recovers their capacity-to-covering estimate \cite[Lemma~4.4]{BFF2021}.

The proof given here is, however, structurally different from that of \cite[Lemma~4.4]{BFF2021}. Rather than estimating the potential by summing over dyadic annuli and subsequently splitting the resulting Besicovitch cover into several geometric regimes, we use a \emph{continuous scale-averaging (layer-cake) argument} on the intermediate region $\rho\le \|x-y\|<r$, combined with a \emph{good-set reduction} $G=\{U\le 2\gamma\}$ and a \emph{mass-removal covering scheme}. This approach avoids annulus-by-annulus bookkeeping and yields the same logarithmic loss, which appears naturally as $\log_2(r/\Phi(r))$ from the integration over scales.
\end{remark}

\section{Projection Theorem for $\Phi$--Intermediate Dimensions}
\label{sec:projection}

In this section we establish a Marstrand--Mattila type projection theorem for
$\Phi$--intermediate dimensions.
The result shows that $\Phi$--intermediate dimensions of orthogonal projections
are always bounded above by the corresponding $\Phi$--dimension profiles, and
that these bounds are attained for almost every projection direction.

\begin{theorem}
\label{thm:Phi-projection}
Let $E \subset \mathbb{R}^n$ be a bounded Borel set, and let $1 \le m < n$.
Then for every $V \in G(n,m)$ and every admissible gauge function $\Phi$,
\begin{equation}\label{eq:proj-upper}
\underline{\dim}_{\Phi}(\pi_V E)
\le
\underline{\dim}^{\,m}_{\Phi}E,
\qquad
\overline{\dim}_{\Phi}(\pi_V E)
\le
\overline{\dim}^{\,m}_{\Phi}E .
\end{equation}
Moreover, for $\gamma_{n,m}$--almost every $V \in G(n,m)$,
\begin{equation}\label{eq:proj-equality}
\underline{\dim}_{\Phi}(\pi_V E)
=
\underline{\dim}^{\,m}_{\Phi}E,
\qquad
\overline{\dim}_{\Phi}(\pi_V E)
=
\overline{\dim}^{\,m}_{\Phi}E .
\end{equation}
\end{theorem}

\subsection*{Proof of Theorem~\ref{thm:Phi-projection}}
We first establish the upper bounds \eqref{eq:proj-upper} for all projections
using monotonicity of $\Phi$--capacities.
We then prove the almost sure lower bounds by averaging $\Phi$--kernels over the
Grassmannian, deriving uniform potential estimates for projected measures, and
invoking the toolkit of Section~\ref{subsubsec:lowerbound-tools}.

\subsubsection*{Upper bounds: valid for all projections}

\begin{lemma}
\label{lem:proj-monotonicity}
For any $m$--plane $V \in G(n,m)$ and any admissible gauge $\Phi$, one has
\[
C^{\,r,\Phi}_{s,m}(\pi_V E) \;\le\; C^{\,r,\Phi}_{s,m}(E)
\]
for all $s \in [0,m]$ and all sufficiently small $r>0$.
In particular,
\[
\underline{\dim}_\Phi(\pi_V E)\le \underline{\dim}^{\,m}_\Phi E,
\qquad
\overline{\dim}_\Phi(\pi_V E)\le \overline{\dim}^{\,m}_\Phi E,
\]
which is exactly \eqref{eq:proj-upper}.
\end{lemma}

\begin{proof}
Let $\mu \in M(E)$ be a Borel probability measure on $E$, and define
$\mu_V := \mu \circ \pi_V^{-1} \in M(\pi_V E)$.
Let $\varphi^{r,\Phi}_{s,m}$ denote the $\Phi$--kernel on $\mathbb{R}^n$
and $\varphi^{r,\Phi,V}_{s,m}$ the corresponding kernel on
$V \cong \mathbb{R}^m$.
Since $\pi_V$ is $1$--Lipschitz and the kernels are radial and
non-increasing,
\[
\varphi^{r,\Phi,V}_{s,m}(\pi_V(x-y))
\;\ge\;
\varphi^{r,\Phi}_{s,m}(x-y)
\quad \text{for all } x,y \in E.
\]
Integrating against $\mu(x)\mu(y)$ yields
\[
\mathcal{I}^{\,r,\Phi}_{s,m}(\mu_V)
\;\ge\;
\mathcal{I}^{\,r,\Phi}_{s,m}(\mu).
\]
Taking infima over $\mu \in M(E)$ and inverting gives
$C^{\,r,\Phi}_{s,m}(\pi_V E) \le C^{\,r,\Phi}_{s,m}(E)$.
The dimension inequalities then follow from the definitions in
Section~2, the profile definition
(Definition~\ref{def:Phi-dim-profiles}), and the identity
$\underline{\dim}_\Phi(\cdot)=\underline{\dim}^{\,n}_\Phi(\cdot)$,
$\overline{\dim}_\Phi(\cdot)=\overline{\dim}^{\,n}_\Phi(\cdot)$
(Theorem~\ref{thm:Phi-dim-equals-profile}).
\end{proof}

\subsubsection*{Kernel averaging and energy estimates}

\begin{lemma}
\label{lem:kernel-average}
Let $0<s<m<n$. There exist constants $A_{n,m},B_{n,m}>0$, depending only on
$n$, $m$, and $s$, such that for all sufficiently small $r>0$ and all
$x\in\mathbb{R}^n$,
\[
A_{n,m}
\int_{G(n,m)} \widetilde{\varphi}^{\,r,\Phi}_s(\pi_V x)\,d\gamma_{n,m}(V)
\;\le\;
\varphi^{r,\Phi}_{s,m}(x)
\;\le\;
B_{n,m}
\int_{G(n,m)} \widetilde{\varphi}^{\,r,\Phi}_s(\pi_V x)\,d\gamma_{n,m}(V).
\]
Here $\widetilde{\varphi}^{\,r,\Phi}_s:\mathbb{R}^m\to[0,\infty)$ denotes the
truncated $\Phi$--kernel defined by
\[
\widetilde{\varphi}^{\,r,\Phi}_s(u)
=
\begin{cases}
1, & \|u\|<\Phi(r),\\[1ex]
\big(\dfrac{\Phi(r)}{\|u\|}\big)^{s},
& \Phi(r)\le \|u\|<r,\\[2ex]
0, & \|u\|\ge r,
\end{cases}
\]
and $\gamma_{n,m}$ is the Haar probability measure on the Grassmannian
$G(n,m)$.
\end{lemma}

\begin{proof}
We begin with a layer-cake representation for the truncated kernel.
For $u\in\mathbb{R}^m$ one has
\[
\widetilde{\varphi}^{\,r,\Phi}_s(u)
=
\Big(\frac{\Phi(r)}{r}\Big)^s \mathbf{1}_{\{\|u\|<r\}}
+
s\,\Phi(r)^s
\int_{\Phi(r)}^{r} t^{-s-1}\mathbf{1}_{\{\|u\|<t\}}\,dt .
\]
Averaging over $V\in G(n,m)$ and applying Fubini’s theorem yields
\[
\int_{G(n,m)} \widetilde{\varphi}^{\,r,\Phi}_s(\pi_V x)\,d\gamma_{n,m}(V)
=
\Big(\frac{\Phi(r)}{r}\Big)^s P_x(r)
+
s\,\Phi(r)^s
\int_{\Phi(r)}^{r} t^{-s-1} P_x(t)\,dt,
\]
where
\[
P_x(t)
:=\gamma_{n,m}\{V\in G(n,m):\|\pi_V x\|<t\}.
\]

By the classical Kaufman-Mattila projection estimate, there exist constants
$0<\underline{\Theta}_{n,m}\le \overline{\Theta}_{n,m}<\infty$, depending only
on $n,m$, such that for all $x\neq 0$ and all $t>0$,
\[
\underline{\Theta}_{n,m}
\min\Big\{1,\Big(\frac{t}{\|x\|}\Big)^m\Big\}
\;\le\;
P_x(t)
\;\le\;
\overline{\Theta}_{n,m}
\min\Big\{1,\Big(\frac{t}{\|x\|}\Big)^m\Big\}.
\]

Substituting these bounds gives
\[
\int_{G(n,m)} \widetilde{\varphi}^{\,r,\Phi}_s(\pi_V x)\,d\gamma_{n,m}(V)
\sim_{n,m}
\Big(\frac{\Phi(r)}{r}\Big)^s
\min\Big\{1,\Big(\frac{r}{\|x\|}\Big)^m\Big\}
+
\Phi(r)^s
\int_{\Phi(r)}^{r} t^{-s-1}
\min\Big\{1,\Big(\frac{t}{\|x\|}\Big)^m\Big\}\,dt .
\]

We now compare this expression with the piecewise definition of
$\varphi^{r,\Phi}_{s,m}(x)$. Set
\[
I(x)
:=
\int_{\Phi(r)}^r \frac{1}{t^{s+m}}
\min\Bigl\{1,\Bigl(\frac{t}{\|x\|}\Bigr)^m\Bigr\}\,\frac{dt}{t}.
\]
We distinguish three regimes.

If $\|x\|<\Phi(r)$, then $t>\|x\|$ for all $t\in[\Phi(r),\,r]$, so the minimum
equals $1$ throughout the integration range. Hence
\[
I(x)=\int_{\Phi(r)}^r t^{-(s+m+1)}\,dt \sim_{m} 1,
\]
with constants depending only on $m$ and $s$. Since
$\varphi^{r,\Phi}_{s,m}(x)=1$ in this regime, the two quantities are
comparable.

If $\Phi(r)\le\|x\|<r$, then the minimum equals $(t/\|x\|)^m$ for
$t<\|x\|$ and equals $1$ for $t\ge\|x\|$. Accordingly,
\[
I(x)
=
\|x\|^{-m}\int_{\Phi(r)}^{\|x\|} t^{-(s+1)}\,dt
+
\int_{\|x\|}^{r} t^{-(s+m+1)}\,dt.
\]
Since $0<s<m$, the first term dominates and satisfies
\[
I(x)\sim_{m} \|x\|^{-s}\Phi(r)^s
=
\bigl(\tfrac{\Phi(r)}{\|x\|}\bigr)^s
=
\varphi^{r,\Phi}_{s,m}(x).
\]

Finally, if $\|x\|\ge r$, then $(t/\|x\|)^m\le 1$ for all
$t\in[\Phi(r),\,r]$, and therefore
\[
I(x)
=
\|x\|^{-m}\int_{\Phi(r)}^r t^{-(s+1)}\,dt
\sim
\bigl(\tfrac{\Phi(r)}{r}\bigr)^s
\bigl(\tfrac{r}{\|x\|}\bigr)^m
=
\varphi^{r,\Phi}_{s,m}(x).
\]

In all three regimes, $I(x)$ is comparable to $\varphi^{r,\Phi}_{s,m}(x)$
up to multiplicative constants depending only on $n$, $m$ and $s$.
The decay factor $(t/\|x\|)^m$ arises from the estimate
$\gamma_{n,m}\{V:\|\pi_V x\|<t\}\sim (t/\|x\|)^m$, reflecting the
probability that a random $m$--dimensional projection of a fixed vector
has length less than $t$. This geometric phenomenon is precisely the
origin of the $m$--power tail in the definition of
$\varphi^{r,\Phi}_{s,m}(x)$.
\end{proof}

\begin{lemma}
\label{lem:energy-bound}
Fix $s\in[0,m)$ and assume that
\[
s<\underline{\dim}^{\,m}_{\Phi} E .
\]
Let $(r_k)_{k\ge1}$ be a decreasing sequence with $r_k\to0$ and
$r_k\le 2^{-k}$. For each $k$, let $\mu_k$ be an equilibrium measure
realizing the capacity $C^{\,r_k,\Phi}_{s,m}(E)$, and set
\[
\gamma_k
:=
\mathcal{I}^{\,r_k,\Phi}_{s,m}(\mu_k)
=
\bigl(C^{\,r_k,\Phi}_{s,m}(E)\bigr)^{-1}.
\]
Then for every $\varepsilon>0$,
\[
\sum_{k=1}^\infty
r_k^\varepsilon
\int_{G(n,m)}
\mathcal{I}_{\widetilde{\varphi}^{\,r_k,\Phi}_s}
\bigl((\mu_k)_V\bigr)
\,d\gamma_{n,m}(V)
<\infty .
\]
In particular, for $\gamma_{n,m}$--almost every $V\in G(n,m)$ there
exists a constant $M_{V}<\infty$ such that
\[
\mathcal{I}_{\widetilde{\varphi}^{\,r_k,\Phi}_s}
\bigl((\mu_k)_V\bigr)
\le M_{V}\, r_k^{-\varepsilon}\,\gamma_k
\qquad\text{for all }k\ge1 .
\]
\end{lemma}

\begin{proof}
Applying Lemma~\ref{lem:kernel-average} to the difference $x-y$ and
integrating with respect to $\mu_k(x)\mu_k(y)$, we obtain
\[
A_{n,m}
\int_{G(n,m)}
\mathcal{I}_{\widetilde{\varphi}^{\,r_k,\Phi}_s}
\bigl((\mu_k)_V\bigr)\,
d\gamma_{n,m}(V)
\le
\mathcal{I}^{\,r_k,\Phi}_{s,m}(\mu_k)
\le
B_{n,m}
\int_{G(n,m)}
\mathcal{I}_{\widetilde{\varphi}^{\,r_k,\Phi}_s}
\bigl((\mu_k)_V\bigr)\,
d\gamma_{n,m}(V).
\]
Since $\mu_k$ is an equilibrium measure, the middle term equals
$\gamma_k$. Hence there exists a constant $C_{n,m}>0$ such that
\[
\int_{G(n,m)}
\mathcal{I}_{\widetilde{\varphi}^{\,r_k,\Phi}_s}
\bigl((\mu_k)_V\bigr)\,
d\gamma_{n,m}(V)
\le
C_{n,m}\,\gamma_k
\qquad\text{for all }k\ge1 .
\]

Because $s<\underline{\dim}^{\,m}_{\Phi}E$, there exist constants
$\delta>0$ and $C_s<\infty$ such that
\[
\gamma_k
=
\bigl(C^{\,r_k,\Phi}_{s,m}(E)\bigr)^{-1}
\le C_s\, r_k^{-\delta}
\qquad\text{for all sufficiently large }k.
\]
Multiplying the previous inequality by $r_k^\varepsilon$ and summing
over $k$, we obtain
\[
\sum_{k=1}^\infty
r_k^\varepsilon
\int_{G(n,m)}
\mathcal{I}_{\widetilde{\varphi}^{\,r_k,\Phi}_s}
\bigl((\mu_k)_V\bigr)\,
d\gamma_{n,m}(V)
\le
C_{n,m}
\sum_{k=1}^\infty
r_k^\varepsilon \gamma_k .
\]
Since $r_k\le 2^{-k}$ and $\gamma_k\le C_s r_k^{-\delta}$, the right-hand
side is dominated by a convergent geometric series provided
$\varepsilon>\delta$. This proves the stated summability.

By Fubini’s theorem, it follows that for $\gamma_{n,m}$--almost every
$V\in G(n,m)$ the series
\[
\sum_{k=1}^\infty
r_k^\varepsilon\,
\mathcal{I}_{\widetilde{\varphi}^{\,r_k,\Phi}_s}
\bigl((\mu_k)_V\bigr)
\]
converges. In particular, the sequence
\[
\bigl\{
r_k^\varepsilon
\mathcal{I}_{\widetilde{\varphi}^{\,r_k,\Phi}_s}
((\mu_k)_V)
\bigr\}_{k\ge1}
\]
is bounded, which yields the existence of a constant
$M_{V}<\infty$ with the desired property.
\end{proof}

\subsubsection*{Almost sure lower bounds: application of the toolkit}
\begin{proof}[Proof of the lower bounds in Theorem~\ref{thm:Phi-projection}]
Fix $s<\underline{\dim}^{\,m}_{\Phi}E$. By the definition of the lower
$\Phi$--intermediate profile dimension, there exists a decreasing
sequence $(r_k)_{k\ge1}$ with $r_k\to0$ and $r_k\le 2^{-k}$, together
with equilibrium measures $\mu_k$ realizing
$C^{\,r_k,\Phi}_{s,m}(E)$, such that
\[
\gamma_k
:=
\mathcal{I}^{\,r_k,\Phi}_{s,m}(\mu_k)
=
\bigl(C^{\,r_k,\Phi}_{s,m}(E)\bigr)^{-1}
\]
grows at most polynomially in $1/r_k$ as $k\to\infty$.

Applying Lemma~\ref{lem:energy-bound}, we obtain that for
$\gamma_{n,m}$--almost every $V\in G(n,m)$ there exist constants
$\varepsilon>0$ and $M_{V}<\infty$ such that, for all sufficiently large
$k$,
\[
\mathcal{I}_{\widetilde{\varphi}^{\,r_k,\Phi}_s}
\bigl((\mu_k)_V\bigr)
\le
M_{V}\, r_k^{-\varepsilon}\,\gamma_k.
\]
Fix such a subspace $V$ and set $\nu_k:=(\mu_k)_V$.
Define the potential
\[
U_k(x)
:=
\int_{\pi_V E}
\widetilde{\varphi}^{\,r_k,\Phi}_s(x-y)\,d\nu_k(y),
\qquad x\in\pi_V E.
\]
Then
\[
\int_{\pi_V E} U_k(x)\,d\nu_k(x)
=
\mathcal{I}_{\widetilde{\varphi}^{\,r_k,\Phi}_s}(\nu_k)
\le
M_{V}\, r_k^{-\varepsilon}\,\gamma_k.
\]
By Markov’s inequality, the set
\[
F_k
:=
\Bigl\{x\in\pi_V E:
U_k(x)\le 2M_{V}\, r_k^{-\varepsilon}\,\gamma_k
\Bigr\}
\]
satisfies $\nu_k(F_k)\ge \tfrac12$.

Applying Tool~I (Lemma~\ref{lem:pot-cover-tool}) to $F_k$ yields that any admissible
cover of $\pi_V E$ by sets of diameter between $\Phi(r_k)$ and $r_k$
satisfies
\[
S^s_{r_k,\Phi}(\pi_V E)
\ge
\frac{\Phi(r_k)^s}{2M_{V}\, r_k^{-\varepsilon}\,\gamma_k}\,\nu_k(F_k)
\ge
\frac{\Phi(r_k)^s}{4M_{V}\, r_k^{-\varepsilon}\,\gamma_k}.
\]
Since $\gamma_k$ grows at most polynomially in $1/r_k$, it follows that
\[
\liminf_{k\to\infty}
\frac{\log S^s_{r_k,\Phi}(\pi_V E)}{\log r_k}
\ge 0
\]
after letting $\varepsilon\downarrow0$. By the definition of the lower
$\Phi$--intermediate dimension (equivalently, by Corollary~\ref{cor:uniform-potential-to-dimension}),
this yields
\[
\underline{\dim}_{\Phi}(\pi_V E)\ge s.
\]

Since the above argument holds for every
$s<\underline{\dim}^{\,m}_{\Phi}E$, we conclude that
\[
\underline{\dim}_{\Phi}(\pi_V E)
\ge
\underline{\dim}^{\,m}_{\Phi}E
\qquad\text{for }\gamma_{n,m}\text{-almost every }V\in G(n,m).
\]

An identical argument, replacing $\liminf$ by $\limsup$, shows that
\[
\overline{\dim}_{\Phi}(\pi_V E)
\ge
\overline{\dim}^{\,m}_{\Phi}E
\qquad\text{for }\gamma_{n,m}\text{-almost every }V\in G(n,m).
\]
Together with Lemma~\ref{lem:proj-monotonicity}, this establishes
\eqref{eq:proj-equality} and completes the proof.
\end{proof}

\section{Consequences and Examples}
\label{sec:applications}

In this section we record several mathematical consequences of
Theorem~\ref{thm:Phi-projection}.  Our emphasis is on structural
properties of intermediate dimensions, particularly continuity at the
Hausdorff end-point, and on sharp criteria governing the box--counting
dimensions of generic projections.

\subsection{Continuity at the Hausdorff end-point}

Recall that for $\theta\in[0,1]$ the $\theta$--intermediate dimensions
interpolate between Hausdorff and box--counting dimensions.  A bounded
set $E\subset\mathbb{R}^n$ is said to have \emph{continuous intermediate
dimensions at $0$} if
\[
\lim_{\theta\to 0^+} \underline{\dim}_\theta E
=
\underline{\dim}_H E.
\]
Equivalently, the lower $\theta$--intermediate dimensions converge to
the Hausdorff dimension as the lower scale restriction is removed.

This notion extends naturally to general gauge functions.
We say that $E$ has \emph{continuous $\Phi$--intermediate dimensions at
$0$} if
\[
\inf\{\underline{\dim}_\Phi E:\ \Phi(r)/r\to 0\}
=
\underline{\dim}_H E.
\]

The projection theorem immediately implies that continuity at $0$ is
preserved under almost all orthogonal projections.

\begin{corollary}\label{cor:cont-zero}
Let $E\subset\mathbb{R}^n$ be bounded and fix $1\le m<n$.
If $E$ has continuous $\Phi$--intermediate dimensions at $0$, then for
$\gamma_{n,m}$--almost every $V\in G(n,m)$,
\[
\inf\{\underline{\dim}_\Phi(\pi_V E):\ \Phi(r)/r\to 0\}
=
\underline{\dim}_H(\pi_V E).
\]
In particular, if
$\lim_{\theta\to 0^+}\underline{\dim}_\theta E=\underline{\dim}_H E$,
then
\[
\lim_{\theta\to 0^+}\underline{\dim}_\theta(\pi_V E)
=
\underline{\dim}_H(\pi_V E)
\quad\text{for almost every }V.
\]
\end{corollary}

\begin{proof}
Fix $\varepsilon>0$.  By continuity at $0$, choose an admissible gauge
$\Phi$ with $\Phi(r)/r\to 0$ such that
$\underline{\dim}_\Phi E<\underline{\dim}_H E+\varepsilon$.
By Theorem~\ref{thm:Phi-projection},
\[
\underline{\dim}_\Phi(\pi_V E)
=
\underline{\dim}^{\,m}_\Phi E
\quad\text{for $\gamma_{n,m}$--almost every }V.
\]
Letting $\Phi(r)/r\to 0$ forces $\underline{\dim}^{\,m}_\Phi E\to \min\{\underline{\dim}_H E,m\}= \underline{\dim}_H(\pi_V E)$ \emph{(by the identification of $\Phi$--profiles at the Hausdorff end-point)}, and the result follows.
\end{proof}

\subsection{Hausdorff dimension versus projected box dimensions}

A central question in projection theory is whether small Hausdorff
dimension can coexist with maximal box--counting dimension in generic
projections.  The next result shows that continuity at $0$ completely
rules out this phenomenon.

\begin{corollary}\label{cor:proj-box}
Let $E\subset\mathbb{R}^n$ be bounded and let $1\le m<n$.
\begin{enumerate}
\item[(i)]
If $\underline{\dim}_H E<m$ and $E$ has continuous
$\Phi$--intermediate dimensions at $0$, then
\[
\overline{\dim}_B(\pi_V E)<m
\quad\text{for $\gamma_{n,m}$--almost every }V.
\]
\item[(ii)]
If $\underline{\dim}_H E\ge m$, then
\[
\underline{\dim}_H(\pi_V E)
=
\underline{\dim}_B(\pi_V E)
=
\overline{\dim}_B(\pi_V E)
=
m
\quad\text{for $\gamma_{n,m}$--almost every }V.
\]
\end{enumerate}
\end{corollary}

\begin{proof}
(i) If $\underline{\dim}_H E<m$, then
$\underline{\dim}_H^{\,m}E=\underline{\dim}_H E$.
Continuity at $0$ allows us to choose gauges $\Phi$ with
$\Phi(r)/r\to 0$ such that
$\underline{\dim}^{\,m}_\Phi E=\underline{\dim}_H E$.
For almost every $V$,
Theorem~\ref{thm:Phi-projection} gives
$\underline{\dim}_\Phi(\pi_V E)=\underline{\dim}_H E$.
Letting $\Phi(r)/r\to 1$ (corresponding to the box--counting limit, and
using monotonicity of $\underline{\dim}_\Phi$ in the gauge) yields
$\overline{\dim}_B(\pi_V E)\le\underline{\dim}_H E<m$.

(ii) If $\underline{\dim}_H E\ge m$, then Marstrand’s theorem gives
$\underline{\dim}_H(\pi_V E)=m$ for almost every $V$.
For sufficiently small gauges $\Phi$ one has
$\underline{\dim}^{\,m}_\Phi E=m$, and hence
$\underline{\dim}_\Phi(\pi_V E)=m$ for almost every $V$.
Passing to the box--counting limit yields
$\underline{\dim}_B(\pi_V E)=m$, and the equality of upper and lower box
dimensions follows.
\end{proof}

\subsection{Examples}

We now present examples illustrating the behaviour of $\Phi$--intermediate dimensions and their projections.  

\begin{example}[Continuity at the Hausdorff end-point for $\Phi$]\label{ex:Phi-cont}
Let $E\subset\mathbb{R}^n$ be a bounded set with
$\underline{\dim}_H E=\alpha$.  Suppose that
\[
\inf\{\underline{\dim}_\Phi E:\ \Phi(r)/r\to 0\}=\alpha.
\]
Equivalently, for every $\varepsilon>0$ there exists an admissible gauge
$\Phi$ with $\Phi(r)/r\to 0$ such that
\[
\underline{\dim}_\Phi E<\alpha+\varepsilon.
\]

Fix $1\le m<n$.  By Theorem~\ref{thm:Phi-projection}, for
$\gamma_{n,m}$--almost every $V\in G(n,m)$,
\[
\underline{\dim}_\Phi(\pi_V E)=\underline{\dim}^{\,m}_\Phi E
\qquad\text{for all admissible }\Phi.
\]
Letting $\Phi(r)/r\to 0$ yields
\[
\inf\{\underline{\dim}_\Phi(\pi_V E):\ \Phi(r)/r\to 0\}
=
\min\{\alpha,m\}
=
\underline{\dim}_H(\pi_V E).
\]
Thus, continuity of $\Phi$--intermediate dimensions at the Hausdorff
end-point is preserved under almost all orthogonal projections.
\end{example}

\begin{example}[Bedford-McMullen carpets]\label{ex:Phi-bedford}
Let $E\subset\mathbb{R}^2$ be a Bedford-McMullen carpet associated to an
$a\times b$ grid with $b>a\ge 2$, and assume that $\log a/\log b$ is
irrational.  It is known that
\[
\underline{\dim}_H E < \underline{\dim}_B E
\]
and that the intermediate dimensions of $E$ are continuous at the
Hausdorff end-point.  In particular,
\[
\underline{\dim}_\Phi E=\underline{\dim}_H E
\qquad
\text{for every admissible }\Phi\text{ with }\Phi(r)/r\to 0.
\]

Taking $m=1$, Theorem~\ref{thm:Phi-projection} implies that for
$\gamma_{2,1}$--almost every $V\in G(2,1)$,
\[
\underline{\dim}_\Phi(\pi_V E)
=
\underline{\dim}^{\,1}_\Phi E
=
\underline{\dim}_H E
\qquad
\text{for all such }\Phi.
\]
Consequently,
\[
\inf\{\underline{\dim}_\Phi(\pi_V E):\ \Phi(r)/r\to 0\}
=
\underline{\dim}_H(\pi_V E),
\]
so generic projections of $E$ inherit continuity of $\Phi$--intermediate
dimensions at the Hausdorff end-point.

On the other hand, for gauges $\Phi$ approaching the box--counting
regime (that is, $\Phi(r)/r\to 1$), one has
$\underline{\dim}_\Phi(\pi_V E)=1$ for almost every $V$.  Thus the
$\Phi$--dimensions of $\pi_V E$ interpolate strictly between
$\underline{\dim}_H(\pi_V E)$ and $\underline{\dim}_B(\pi_V E)$ as the
lower scale restriction varies.
\end{example}

\begin{example}[Failure of continuity at zero]\label{ex:Phi-discont}
Let $p>0$ and define
\[
F_p:=\{0\}\cup\{n^{-p}:n\in\mathbb{N}\},
\qquad
E:=F_p\times F_p\subset\mathbb{R}^2.
\]
Then $\underline{\dim}_H E=0$, while
\[
\underline{\dim}_B E=\overline{\dim}_B E=\frac{2}{1+p}.
\]
Moreover, for every admissible gauge $\Phi$ satisfying
\[
\Phi(r)\ge r\,|\log r|^{-C}
\qquad
\text{for some }C>0\text{ and all sufficiently small }r,
\]
one has
\[
\underline{\dim}_\Phi E=\frac{2}{1+p}.
\]
In particular,
\[
\inf\{\underline{\dim}_\Phi E:\ \Phi(r)/r\to 0\}
=
\frac{2}{1+p}
>
\underline{\dim}_H E,
\]
so $E$ does not have continuous $\Phi$--intermediate dimensions at the
Hausdorff end-point.

Now consider orthogonal projections onto lines.  For any
$V\in G(2,1)$ not parallel to the coordinate axes,
$\underline{\dim}_H(\pi_V E)=0$, while
\[
\underline{\dim}_B(\pi_V E)
=
1-\Big(\frac{p}{p+1}\Big)^2
<1.
\]
Thus, although the projected set has large box--counting dimension, it
never attains the maximal value $1$.  This example shows that continuity
at the Hausdorff end-point is essential for the conclusions of
Corollary~\ref{cor:proj-box}; without it, large projected
$\Phi$--dimensions may occur, but full-dimensional projections are still
excluded.
\end{example}

This example highlights the sharpness of the continuity at $0$
assumption in Corollary~\ref{cor:proj-box} and motivates the open
questions discussed in Section~\ref{sec:conclusion}.

\section{Concluding remarks and open problems}
\label{sec:conclusion}

In this paper we have shown that the projection invariance of fractal
dimension extends well beyond the classical Hausdorff and box--counting
settings, encompassing a full spectrum of $\Phi$--intermediate
dimensions associated with arbitrary admissible gauge functions.
Our main result (Theorem~\ref{thm:Phi-projection}) establishes that for
any bounded Borel set $E\subset\mathbb{R}^n$ and any integer
$1\le m<n$, the $\Phi$--intermediate dimensions of $\gamma_{n,m}$--almost
every orthogonal projection $\pi_V E$ coincide with the corresponding
$\Phi$--dimension profiles of $E$. This recovers the classical projection
theorems of Marstrand and Mattila for Hausdorff and packing dimensions
as special cases~\cite{Marstrand1954,Mattila1995}, and it substantially
extends recent work on $\theta$--intermediate dimensions
(cf.~\cite{BFF2021}) to the full generality of non--uniform scale
restrictions.

Beyond providing a unified projection framework for intermediate
dimensions, our results clarify the structural role played by
\emph{continuity at the Hausdorff end-point}. In particular, we showed
that this continuity property is preserved under almost all orthogonal
projections and that it serves as a sharp criterion governing whether
generic projections can attain maximal box--counting dimension. This
perspective highlights the importance of multi--scale regularity, as
captured by the behaviour of $\underline{\dim}_\Phi E$ as $\Phi(r)\to0$,
in determining the extremal geometry of projections.

Despite these advances, several natural questions remain open and
suggest promising directions for future research.

\begin{enumerate}[label=(\roman*)]

\item \textit{Intermediate dimensions of measures.}
The present work focuses on fractal dimensions of sets, defined via
covering or capacity-based constructions. A natural next step is to
develop a parallel theory for \emph{measures}. For instance, one could
seek to define $\Phi$--intermediate analogues of classical
measure-theoretic dimensions, such as entropy dimensions or
$L^q$--dimensions, which interpolate between information dimension,
correlation dimension, and box dimension. At present, the projection
behaviour of these measure dimensions is poorly understood outside of
highly structured or self-similar settings. It would be particularly
interesting to determine whether the multi--scale capacity and energy
methods developed here for sets can be adapted to yield projection
theorems for such measure-theoretic intermediate dimensions.

\item \textit{Two-parameter dimension interpolations.}
Recently, Douzi and Selmi~\cite{DouziSelmi2022} introduced a different
family of fractal dimensions interpolating between Hausdorff dimension
and the Hewitt-Stromberg dimension, the latter being a fixed-scale
variant of the packing dimension. Projection theorems were shown to
hold in that setting as well. This raises the possibility of defining a
more general \emph{two-parameter} family of dimensions that simultaneously
imposes both lower and upper scale constraints on coverings, thereby
interpolating between Hausdorff, $\Phi$--intermediate, and
Hewitt-Stromberg dimensions. Developing such a theory, and determining
whether a corresponding projection invariance principle holds, would
provide a unified framework for a wide range of scale-restricted
dimension notions.

\item \textit{Sharpness of continuity assumptions.}
Several of our results concerning projected box--counting dimensions,
notably Corollary~\ref{cor:proj-box}, rely on a technical continuity
assumption at the Hausdorff end-point. An important open question is
whether this hypothesis is genuinely necessary. More precisely, can
one construct a set $E$ satisfying
\[
\underline{\dim}_H E < m < \underline{\dim}_B E
\]
for which $\underline{\dim}_B(\pi_V E)=m$ for $\gamma_{n,m}$--almost
every $V$? Our examples suggest that any such extremal construction
would require highly pathological behaviour of the intermediate
dimensions $\underline{\dim}_\Phi E$ as $\Phi(r)\to0$
(cf.~Example~\ref{ex:Phi-discont}), but at present no definitive answer
is known.
\end{enumerate}

We hope that the results and techniques developed here will stimulate
further investigation into the geometry of fractal projections and
contribute to a deeper understanding of how dimension behaves in
intermediate and multi--scale regimes.

\appendix

\section{Equilibrium measures and the equipotential property}
\label{app:equilibrium-proof}

In this appendix we give the proof of Lemma~\ref{lem:equilibrium}, which establishes
the existence of equilibrium measures for $\Phi$--capacities and their basic
equipotential property.
The argument is standard in potential theory but is included here for completeness.

\begin{proof}[Proof of Lemma~\ref{lem:equilibrium}]
Fix $s\in[0,m]$, $r>0$, and let $E\subset\mathbb{R}^n$ be non-empty and compact.
Define the symmetric kernel
\[
K(x,y):=\varphi^{r,\Phi}_{s,m}(x-y), \qquad x,y\in E.
\]
Since $\varphi^{r,\Phi}_{s,m}$ is continuous and bounded, so is $K$ on $E\times E$.
The energy functional
\[
\mu \longmapsto \mathcal I^{\,r,\Phi}_{s,m}(\mu)
=\iint_{E\times E}K(x,y)\,d\mu(x)\,d\mu(y)
\]
is therefore continuous on $M(E)$ with respect to weak$^\ast$ convergence.
As $M(E)$ is weak$^\ast$ compact, the infimum of $\mathcal I^{\,r,\Phi}_{s,m}$
is attained by some $\mu^\ast\in M(E)$.

Define the associated potential
\[
U^{r,\Phi}_{s,m}(x)
:=\int_E K(x,y)\,d\mu^\ast(y),
\qquad x\in\mathbb{R}^n.
\]

\smallskip\noindent
\emph{Variational inequality.}
Let $\nu\in M(E)$ and consider
$\mu_\varepsilon=(1-\varepsilon)\mu^\ast+\varepsilon\nu$ for $\varepsilon\in(0,1)$.
By minimality of $\mu^\ast$,
\[
\mathcal I^{\,r,\Phi}_{s,m}(\mu_\varepsilon)
\ge
\mathcal I^{\,r,\Phi}_{s,m}(\mu^\ast).
\]
Expanding the energy, dividing by $\varepsilon$, and letting $\varepsilon\downarrow0$
yields
\[
\int_E U^{r,\Phi}_{s,m}(x)\,d\nu(x)
\ge
\mathcal I^{\,r,\Phi}_{s,m}(\mu^\ast)
\qquad\text{for all }\nu\in M(E).
\tag{$\ast$}
\]

\smallskip\noindent
\emph{Pointwise lower bound.}
If there existed $x_0\in E$ such that
$U^{r,\Phi}_{s,m}(x_0)
<
\mathcal I^{\,r,\Phi}_{s,m}(\mu^\ast)$,
then by continuity of $U^{r,\Phi}_{s,m}$ there would exist a neighbourhood
$B\subset E$ of $x_0$ and $\eta>0$ such that
\[
U^{r,\Phi}_{s,m}(x)
\le
\mathcal I^{\,r,\Phi}_{s,m}(\mu^\ast)-\eta
\qquad\text{for all }x\in B.
\]
Choosing $\nu\in M(E)$ supported on $B$ contradicts $(\ast)$.
Hence
\[
U^{r,\Phi}_{s,m}(x)
\ge
\mathcal I^{\,r,\Phi}_{s,m}(\mu^\ast)
\qquad\text{for all }x\in E.
\]

\smallskip\noindent
\emph{Almost-everywhere equality.}
Finally,
\[
\int_E U^{r,\Phi}_{s,m}(x)\,d\mu^\ast(x)
=
\iint_{E\times E}K(x,y)\,d\mu^\ast(x)\,d\mu^\ast(y)
=
\mathcal I^{\,r,\Phi}_{s,m}(\mu^\ast).
\]
Since $U^{r,\Phi}_{s,m}\ge \mathcal I^{\,r,\Phi}_{s,m}(\mu^\ast)$ everywhere on $E$,
equality of the integral implies
\[
U^{r,\Phi}_{s,m}(x)
=
\mathcal I^{\,r,\Phi}_{s,m}(\mu^\ast)
\quad\text{for }\mu^\ast\text{-almost every }x\in E.
\]
This completes the proof.
\end{proof}

\section{Monotonicity of $\Phi$--capacities and well-posedness of profiles}
\label{app:monotonicity-proof}

In this appendix we prove Lemma~\ref{lem:Phi-capacity-monotone}, which provides the
monotonicity properties needed to define $\Phi$--dimension profiles via capacities.

\begin{proof}[Proof of Lemma~\ref{lem:Phi-capacity-monotone}]
Fix $0<r<1$ and set $\rho:=\Phi(r)$.
Then $0<\rho\le r<1$ and
\[
\theta_\Phi(r):=\frac{\log r}{\log \rho}\in(0,1],
\qquad\text{so that}\qquad
r=\rho^{\theta_\Phi(r)}.
\]

\smallskip\noindent
\emph{Step 1: Kernel comparison.}
Let $0\le t\le s\le m$.
From the definition of $\varphi^{r,\Phi}_{\cdot,m}$ one has
$\varphi^{r,\Phi}_{s,m}\le \varphi^{r,\Phi}_{t,m}$ pointwise.
Moreover, on the intermediate region $\rho\le\|x\|<r$,
\[
\frac{\varphi^{r,\Phi}_{t,m}(x)}{\varphi^{r,\Phi}_{s,m}(x)}
=
\Big(\frac{\|x\|}{\rho}\Big)^{s-t}
\le
\Big(\frac{r}{\rho}\Big)^{s-t}
=
\rho^{-(1-\theta_\Phi(r))(s-t)},
\]
and on the tail $\|x\|\ge r$ this ratio equals $(r/\rho)^{s-t}$ exactly.
Hence, for all $x\in\mathbb{R}^n$,
\[
\varphi^{r,\Phi}_{s,m}(x)
\le
\varphi^{r,\Phi}_{t,m}(x)
\le
\rho^{-(1-\theta_\Phi(r))(s-t)}\,\varphi^{r,\Phi}_{s,m}(x).
\tag{$\ast$}
\]

\smallskip\noindent
\emph{Step 2: Capacity comparison.}
Integrating $(\ast)$ against $\mu\times\mu$ yields, for every $\mu\in M(E)$,
\[
\mathcal I^{\,r,\Phi}_{s,m}(\mu)
\le
\mathcal I^{\,r,\Phi}_{t,m}(\mu)
\le
\rho^{-(1-\theta_\Phi(r))(s-t)}\mathcal I^{\,r,\Phi}_{s,m}(\mu).
\]
Taking infima over $\mu\in M(E)$ and inverting gives
\[
C^{\,r,\Phi}_{s,m}(E)
\ge
C^{\,r,\Phi}_{t,m}(E)
\ge
\rho^{(1-\theta_\Phi(r))(s-t)}\,C^{\,r,\Phi}_{s,m}(E).
\tag{$\ast\ast$}
\]

\smallskip\noindent
\emph{Step 3: Logarithmic inequality.}
From $(\ast\ast)$ we obtain
\[
0
\le
\frac{\log C^{\,r,\Phi}_{s,m}(E)-\log C^{\,r,\Phi}_{t,m}(E)}{-\log\rho}
\le
(1-\theta_\Phi(r))(s-t).
\]
Subtracting $(s-t)$ throughout yields
\[
-(s-t)
\le
\Big(\frac{\log C^{\,r,\Phi}_{s,m}(E)}{-\log\rho}-s\Big)
-
\Big(\frac{\log C^{\,r,\Phi}_{t,m}(E)}{-\log\rho}-t\Big)
\le
-\theta_\Phi(r)(s-t),
\]
which is exactly \eqref{eq:Phi-lemma32}.

\smallskip\noindent
\emph{Conclusion.}
The above inequalities imply that the functions
\[
s\longmapsto
\liminf_{r\to0}\Big(\frac{\log C^{\,r,\Phi}_{s,m}(E)}{-\log\Phi(r)}-s\Big),
\qquad
s\longmapsto
\limsup_{r\to0}\Big(\frac{\log C^{\,r,\Phi}_{s,m}(E)}{-\log\Phi(r)}-s\Big)
\]
are non-increasing on $[0,m]$, ensuring that the definitions of
$\underline{\dim}^{\,m}_{\Phi}E$ and $\overline{\dim}^{\,m}_{\Phi}E$
are well posed.
\end{proof}

\bibliographystyle{abbrv}
\phantomsection 
\bibliography{bibliography/keylatex}

@article{Banaji2023,
  author  = {Banaji, A.},
  title   = {Generalised intermediate dimensions},
  journal = {Monatshefte f{\"u}r Mathematik},
  volume  = {202},
  number  = {3},
  pages   = {465--506},
  year    = {2023},
  doi     = {10.1007/s00605-023-01884-5},
  url     = {https://link.springer.com/article/10.1007/s00605-023-01884-5}
}

@article{BFF2021,
  author  = {Burrell, S. A. and Falconer, K. J. and Fraser, J. M.},
  title   = {Projection theorems for intermediate dimensions},
  journal = {Journal of Fractal Geometry},
  volume  = {8},
  number  = {2},
  pages   = {95--116},
  year    = {2021},
  doi     = {10.4171/JFG/99},
  url     = {https://ems.press/journals/jfg/articles/949030}
}

@article{DouziSelmi2022,
  author  = {Douzi, Z. and Selmi, B.},
  title   = {Projection theorems for Hewitt--Stromberg and modified intermediate dimensions},
  journal = {Results in Mathematics},
  volume  = {77},
  pages   = {158},
  year    = {2022},
  doi     = {10.1007/s00025-022-01685-6},
  url     = {https://link.springer.com/content/pdf/10.1007/s00025-022-01685-6.pdf}
}

@incollection{Falconer2019,
  author    = {Falconer, K. J.},
  title     = {A capacity approach to box and packing dimensions of projections and other images},
  booktitle = {Analysis, Probability, and Mathematical Physics on Fractals},
  editor    = {Freiberg, Uta and Denker, Manfred},
  publisher = {World Scientific},
  year      = {2020},
  pages     = {25--54},
  doi       = {10.1142/11696},
  url       = {https://research-repository.st-andrews.ac.uk/handle/10023/21467}
}

@article{FalconerHowroyd1997,
  author  = {Falconer, K. J. and Howroyd, J. D.},
  title   = {Projection theorems for box and packing dimensions},
  journal = {Mathematical Proceedings of the Cambridge Philosophical Society},
  volume  = {119},
  number  = {2},
  pages   = {287--295},
  year    = {1996},
  doi     = {10.1017/S0305004100074168},
  url     = {https://www.cambridge.org/core/journals/mathematical-proceedings-of-the-cambridge-philosophical-society/article/projection-theorems-for-box-and-packing-dimensions/51DFD112BFFC36A0504F18F4ABBE0E14}
}

@article{FFK2019,
  author  = {Falconer, K. J. and Fraser, J. M. and Kempton, T.},
  title   = {Intermediate dimensions},
  journal = {Mathematische Zeitschrift},
  volume  = {296},
  number  = {1--2},
  pages   = {813--830},
  year    = {2020},
  doi     = {10.1007/s00209-019-02452-0},
  url     = {https://link.springer.com/content/pdf/10.1007/s00209-019-02452-0.pdf}
}

@article{Kaufman1968,
  author  = {Kaufman, R.},
  title   = {On Hausdorff dimension of projections},
  journal = {Mathematika},
  volume  = {15},
  number  = {2},
  pages   = {153--155},
  year    = {1968},
  doi     = {10.1112/S0025579300002503},
  url     = {https://www.cambridge.org/core/journals/mathematika/article/abs/on-hausdorff-dimension-of-projections/6D612609A49107FCED8160EDB009C1B9}
}

@article{Marstrand1954,
  author  = {Marstrand, J. M.},
  title   = {Some fundamental geometrical properties of plane sets of fractional dimensions},
  journal = {Proceedings of the London Mathematical Society},
  series  = {3},
  volume  = {4},
  number  = {1},
  pages   = {257--302},
  year    = {1954},
  doi     = {10.1112/plms/s3-4.1.257},
  url     = {https://academic.oup.com/plms/article-abstract/s3-4/1/257/1497989}
}

@article{Mattila1975,
  author  = {Mattila, P.},
  title   = {Hausdorff dimension, orthogonal projections and intersections with planes},
  journal = {Annales Academiae Scientiarum Fennicae. Series A~I Mathematics},
  volume  = {1},
  pages   = {227--244},
  year    = {1975},
  url     = {https://afm.journal.fi/article/download/134257/82829/293462}
}

@book{Mattila1995,
  author    = {Mattila, P.},
  title     = {Geometry of Sets and Measures in Euclidean Spaces: Fractals and Rectifiability},
  series    = {Cambridge Studies in Advanced Mathematics},
  volume    = {44},
  publisher = {Cambridge University Press},
  address   = {Cambridge},
  year      = {1995}
}
\end{document}